\documentclass[a4paper, 12pt]{amsart}

\usepackage[french, english]{babel}
\usepackage[normalem]{ ulem }
\usepackage{soul}

\usepackage[T1]{fontenc}
\usepackage[utf8]{inputenc}
\usepackage{amsmath,amsfonts,amssymb,amsopn,amscd,amsthm}
\usepackage{comment}
\usepackage{dsfont}
\usepackage{graphicx}
\usepackage{color}
\usepackage[colorlinks]{hyperref}
\usepackage{epigraph}
\usepackage{enumerate}

\usepackage{bookmark}

\usepackage{marginnote}
\usepackage{todonotes}
\usepackage[left=3.5cm, top=2.5cm, bottom=2cm, right=4cm, marginparwidth=3cm]{geometry}

\title[ Matsumoto--Yor on Jordan algebras ]
      { Matsumoto--Yor processes on Jordan algebras }

\author[Chhaibi]{\textsc{Reda Chhaibi}}
\address{R. C. : Universit\'e C\^ote d'Azur, CNRS, LJAD\\
Parc Valrose\\
06108 Nice Cedex 02, France}
\email{{\tt reda.chhaibi@univ-cotedazur.fr}}

\author[Defosseux]{\textsc{Manon Defosseux}} 
\address{M. D. : Universit\'e Paris Cité, MAP5, 45 Rue des Saints-Pères, 75006, Paris, France}
\email{{\tt manon.defosseux@parisdescartes.fr}}

\date{\today}

\allowdisplaybreaks[4]

\DeclareMathOperator{\Ad}{Ad}

\DeclareMathOperator{\tr}{tr}

\DeclareMathOperator{\eqlaw}{\stackrel{\Lc}{=}}

\DeclareMathOperator{\id}{id}
\DeclareMathOperator{\Tr}{Tr}

\DeclareMathOperator{\Lrm}{\mathrm L}
\DeclareMathOperator{\Prm}{\mathrm P}

\def\half{\frac{1}{2}}

\def\1{{\mathbf 1}}

\def\b{\beta}

\def\C{{\mathbb C}}
\def\H{{\mathbb H}}
\def\N{{\mathbb N}}

\def\R{{\mathbb R}}
\def\S{{\mathbb S}}

\def\P{{\mathbb P}}
\def\E{{\mathbb E}}
\def\F{{\mathbb F}}

\def\Ac{{\mathcal A}}

\def\Dc{{\mathcal D}}

\def\Fc{{\mathcal F}}

\def\Lc{{\mathcal L}}

\def\Oc{{\mathcal O}}

\def\afrak{{\mathfrak a}}

\def\gfrak{{\mathfrak g}}

\def\kfrak{{\mathfrak k}}
\def\nfrak{{\mathfrak n}}

\def\pfrak{{\mathfrak p}}

\numberwithin{equation}{section}

\newtheorem{thm}{Theorem}[section]
\newtheorem{proposition}[thm]{Proposition}

\newtheorem{definition}[thm]{Definition}

\newtheorem{lemma}[thm]{Lemma}

\newtheorem{rmk}[thm]{Remark}

\newcommand{\ba}{\begin{array}}
\newcommand{\ea}{\end{array}}
\newcommand{\be}{\begin{equation}}
\newcommand{\ee}{\end{equation}}
\newcommand{\bea}{\begin{eqnarray}}
\newcommand{\eea}{\end{eqnarray}}
\newcommand{\beaa}{\begin{eqnarray*}}
\newcommand{\eeaa}{\end{eqnarray*}}
\newcommand{\ignore}[1]{}
\newcommand{\vertiii}[1]{{\left\vert\kern-0.25ex\left\vert\kern-0.25ex\left\vert #1 
    \right\vert\kern-0.25ex\right\vert\kern-0.25ex\right\vert}}
    
\begin{document}

\begin{abstract}
  The process $(\int_0^t e^{2b_s-b_t}\, ds\ ;\ t\ge 0)$, where $b$ is a real Brownian motion, is known as the geometric 2M-X Matsumoto--Yor process. Remarkably, it enjoys the Markov property. We provide a generalization of this process in the context of Jordan algebras, and we prove the Markov property for this generalization.
 
  Our Markov process occurs as a limit of discrete-time AX+B Markov chains on the cone of squares whose invariant probability measures classically yield a Dufresne-type identity for a perpetuity. In particular, the paper provides a generalization to any symmetric cone of the matrix--valued generalization of the Matsumoto--Yor process and Dufresne identity initially developed by Rider--Valk\'o.
\end{abstract} 

\maketitle 

\setcounter{tocdepth}{2}
\hrule 
\tableofcontents
\hrule

\section{Introduction}
\subsection{Background and literature}
For any stochastic process $X$, we denote its filtration by $\F^X = \left( \Fc^X_t \ ; \ t \geq 0 \right)$.

Consider a real number $\nu$ and a standard real Brownian motion $b = (b_t \ ; \ t\ge 0)$. We denote by $b^{(\nu)}$ this same Brownian motion with an added drift $\nu$. Pitman's theorem \cite{pitman1975one, rogers1981markov} is a seminal result that has attracted significant attention in recent decades. 

\begin{thm}[Pitman's Theorem \cite{pitman1975one}]
\label{thm:pitman}
The process $\lambda = \left( \lambda_t \ ; \ t \geq 0 \right)$ defined for $t \geq 0$ by
$$
   \lambda_t := b_t^{(\nu)} - 2 \inf_{0 \leq s \leq t} b_s^{(\nu)} \ ,
$$
is a Markov process, whose natural filtration is strictly contained in that of $b$, i.e. for all $t > 0$, $\Fc_t^\lambda \varsubsetneq \Fc_t^b$.

Furthermore, we have the conditional distribution
$$
   \P\left( b^{(\nu)} \in dx \ | \ \Fc_t^\lambda, \ \lambda_t = \lambda \right) 
   = \frac{\nu}{2\sinh(\nu \lambda)} \exp\left( \nu x \right) \mathds{1}_{\{ -\lambda \leq x \leq \lambda \}} dx \ .
$$
\end{thm}

There are multiple points of view for understanding this theorem. These perspectives can be divided into two categories: combinatorial and probabilistic points of view where the group $SL_2$ is not involved, and the geometric or representation-theoretic points of view where $SL_2$ plays a crucial role.

\medskip

In \cite{matsumoto2000analogue} and \cite{matsumoto2001analogue} Hiroyuki Matsumoto and Marc Yor give a geometric version of Pitman's theorem. The Generalized Inverse Gaussian (GIG) distribution on $\R_+^*$ appears naturally in that setting. The definition of the $\mathrm{GIG}_{\R_+^*}(\nu ; a, b)$ distribution is \cite[Eq. (9.1)]{matsumoto2001analogue}
\begin{align}
\label{def:GIG}
\mathrm{GIG}_{\R_+^*}(\nu ; a, b)(dx) & =
\frac{\left( \frac{b}{a} \right)^\nu x^\nu}{2 K_\nu(ab)}
\exp\left( -\half \left( a x + \frac{b}{x} \right) \right) \mathds{1}_{\{x \in \R_+^* \}} \frac{dx}{x} \ .
\end{align}
Here the normalization constant is given by $K_\nu$, the Bessel $K$-function, also known as the Macdonald function.
 
\begin{thm}[Matsumoto--Yor \cite{matsumoto2000analogue}]
\label{thm:matsutomo_yor}
The process $Z=(Z_t \ ; \ t\ge 0)$ defined by
$$Z_t := \exp( b_t^{(\nu)} )\int_{0}^{t} \exp\left( -2 b_s^{(\nu)} \right) ds \ , \ t\ge 0,$$
is a Markov process, whose natural filtration is strictly contained in that of $b$, i.e. for all $t > 0$, $\Fc_t^Z \varsubsetneq \Fc_t^b$.

Furthermore, we have the conditional distribution
$$
   \P\left( \exp\left( b^{(\nu)} \right) \in dx \ | \ \Fc_t^Z, \ Z_t = z \right) 
   = \mathrm{GIG}_{\R_+^*}\left( \nu ; \frac1z , \frac1z \right)(dx) \ .
$$ 
\end{thm}
This theorem is a generalization of Pitman's Theorem \ref{thm:pitman} by considering the process $\left( \frac{1}{c}\log Z_{c^2t} \ ; \ t \geq 0 \right)$ for $c>0$ and taking the limit $c \rightarrow 0$. The Pitman transform is recovered as a limiting case by Brownian scaling and the Laplace method. Also, in the proof of Theorem \ref{thm:matsutomo_yor}, an important role is played by the  integral process $\iota = (\iota_t \ ; \ t\ge 0)$ and the Markov process $\ell = ( \ell_t \ ; \ t\ge 0)$ defined by 
$$
  \iota_t :=  \int_0^t \exp\left( -2 b_s^{(\nu)} \right) ds \ ,
  \textrm{ and }\,
  \ell_t := \exp\left( 2 b_t^{(\nu)} \right) \int_0^t \exp\left( -2 b_s^{(\nu)} \right) ds \ .
$$
The integral process arises in several contexts, in particular in mathematical finance. In this context, Daniel Dufresne \cite{dufresne1990distribution} established that if $\nu>0$, then
the perpetuity 
$$ 2 \int_0^\infty \exp\left( -2 b_s^{(\nu)} \right) ds \overset{\Lc}
   =
   \frac{1}{\gamma_\nu} \ ,$$ 
in which $\gamma_\nu$ has the Gamma($\nu$) distribution. The Brownian scaling and the Laplace method can be used to rescale the integral process $\iota$ and the Markov process $\ell$ respectively. This results in the infimum process and in the Lévy tranform process of the Brownian motion. Of course, the choice of letter $\ell$ stands for ``Lévy'' and our naming conventions shall stand throughout the paper.

\medskip

In \cite{chhaibi2013littelmann}, the first author gives a geometric or representation-theoretic interpretation of the Matsumoto--Yor process where $SL_2$ plays a key role. Actually a major role is played by the left-invariant process taking values in $\H^3 \approx SL_2(\C) / SU_2 $ defined by
\begin{align}
    d B_t = B_t \circ
    \begin{pmatrix}
        db_t^{(\nu)} & 0 \\
        dt  & -db_t^{(\nu)}
    \end{pmatrix},
    \ 
\end{align}
which is readily solved in the form
$$ 
  B_t = 
  \begin{pmatrix}
    \exp\left( b_t^{(\nu)} \right) & 0 \\
    Z_t & \exp\left( -b_t^{(\nu)} \right)
\end{pmatrix}
\ .
$$
The role of the hyperbolic space $\H^3$ was further revealed in \cite{chapon2021quantum}, together with the representation theory of the quantum group associated to $SL_2$. A simpler yet remarkable fact is that $B_t = N_t A_t = A_t \widetilde{N}_t$, where $N = \left( N_t \ ; \ t\ge 0 \right)$ and $\widetilde{N} = \left( \widetilde{N}_t \ ; \ t\ge 0 \right)$ are lower triangular matrix processes containing respectively the integral process $\iota$ and the Markov process $\ell$.

\medskip

In \cite{rider2016matrix}, Brian Rider and Benedek Valk\'o introduce a Matsumoto--Yor matrix process which involves a Brownian motion on the general linear group $GL_r(\R)$ for $r\ge 1$. To state their theorem, let us consider a standard left Brownian motion $(g_t \ ; \ t\ge 0)$ on $GL_r(\R)$ with drift $\nu \id_{M_r(\R)}$, for $\nu \in \R$, where the set of $r\times r$ real matrices is denoted by $M_r(\R)$. The subset  of positive definite symmetric real matrices is denoted by $S_r^+(\R)$. The generalization of GIG distributions takes the form
\begin{align}
\label{def:GIG_matrix}
  & \ \mathrm{GIG}_{S_r^+(\R)}(\nu ; a, b)(dx) \\
= & \ \frac{ \det(x)^{\nu} }{2 K\left( \nu ;  a, b \right)}
   \exp\left( -\frac{1}{2}\Tr( a x + b x^{-1}) 
       \right) 
\mathds{1}_{\{x \in S_r^+(\R) \}} 
\frac{dx}{\det(x)^{ \frac{r+1}{2}}} \ ,
\nonumber
\end{align}
where $a$, $b$ are this time matrices in $S_r^+(\R)$. Here, up to normalization, $dx$ is the Lebesgue measure on the Euclidean space of symmetric matrices and $dx / {\det(x)^{ \frac{r+1}{2}}}$ is the conjugation invariant measure. The normalizing constant is given by the function $K\left( \nu ; a, b \right)$ which belongs to the Bessel-like family introduced by Carl Herz in 1955 \cite{herz1955bessel}. See \cite{terras2012harmonic} for further references on the role of this special function. 

\begin{thm}[Rider--Valk\'o \cite{rider2016matrix}] 
\label{thm:VR}
The process $Z = (Z_t \ ; \ t\ge 0)$ defined for $t \geq 0$ by
$$ Z_t := \left(\int_0^t (g_s g_s^*)^{-1} ds
          \right) g_t \ ,$$
is a Markov process, whose natural filtration is strictly contained in that of $g$, i.e. for all $t > 0$, $\Fc_t^Z \varsubsetneq \Fc_t^g$.

Furthermore, the conditional law of $g^*_t Z_t$ given $\left( Z_s \ ; \ s \le t \right)$ is a matrix GIG distribution on $S_r^+(\mathbb R)$ given by
\begin{align*}
     \ \P\left( g^*_t Z_t \in dx \ | \ \Fc_t^Z, \ Z_t = z \right)
 = & \ \mathrm{GIG}_{S_r^+(\R)}(\nu ; (z^*z)^{-1}, \id)(dx) \ .
\end{align*}
\end{thm}
\begin{proof}[Pointers to proof]
The Markov property appears as \cite[Theorem 5]{rider2016matrix}, along with the expression of the infinitesimal generator.
The definition of the matrix GIG distribution is given as \cite[Eq. (18)]{rider2016matrix}. And the conditional distribution is stated at the beginning of \cite[page 182]{rider2016matrix}.
\end{proof}
Notice that indeed the Matsumoto--Yor Theorem \ref{thm:matsutomo_yor} is recovered when setting $r=1$ in the Rider--Valk\'o Theorem \ref{thm:VR}. Less obvious is the fact that the theorem and its proof imply in particular that the ``squared'' process $\left(Z^*_tZ_t \ ; \ t\ge 0\right)$ is Markovian on $S_r^+(\R)$. And based on the conditional distribution, it can be argued that the process $Z^* Z$ is more fundamental. Moreover, Rider and Valk\'o establish a matrix Dufresne identity, demonstrating that for $\nu>(r-1)/2$, the integral
\begin{align}
    \label{eq:valko_rider_perpetuity}
    \int_0^\infty (g_s g_s^*)^{-1} ds
\end{align}
is finite almost surely and distributed according to an inverse matrix Wishart distribution with parameter $2\nu$.

In passing let us mention the similarities between the Dufresne identity and the Bougerol identity \cite{bougerol1983exemples}. Namely, given two independent Brownian motions $\b^{(\nu)}$ and $\gamma^{(\mu)}$ with respective drifts $\nu>0$ and $\mu$, a variant of the Bougerol identity is that
$$
  \int_0^\infty e^{-b_s^{(\nu)}} d\gamma^{(\mu)}_s
$$
has density with respect to Lebesgue measure 
$$ \frac{ e^{-2\mu \arctan(x) } }{ (1 + x^2)^{\nu+ \half} } \ .$$
When $\mu=0$ and conditionally to $b^{(\nu)}$, the above random variable is a Gaussian with variance $\int_0^\infty e^{-2b_s^{(\nu)}} ds$, hence a clear relationship with the Dufresne identity.
The connection is further explored in \cite{assiotis2018matrix}, where Theodoros Assiotis introduces a matrix generalization of the Bougerol identity involving Hua-Pickrell measures. As Assiotis elaborates, his construction parallels the Rider–Valk\'o matrix generalization of the Dufresne identity involving inverse Wishart distributions \eqref{eq:valko_rider_perpetuity}.

In \cite{arista2024Matsumoto} Jonas Arista, Elia Bisi and Neil O'Connell introduce a discrete time version of the Rider--Valk\'o's process. Their results rest in particular on a Lukacs--Olkin--Rubin characterization of Wishart distributions on the set of positive definite symmetric real matrices \cite{casalis1996lukacs}. In \cite{herent2024discrete} Charlie H\'erent studies a discrete time version of the Matsomoto--Yor's process inspired from the approach of \cite{chhaibi2013littelmann} where the first author shows that the generalized inverse Gaussian distributions provide canonical laws on the subsets of $SL_2(\R)$ of lower triangular matrices with a fixed lower left corner. These distributions and Lie-theoretic generalization play a crucial role in \cite{chhaibi2013littelmann}, where early instances of discrete time Matsumoto--Yor processes appear in the general context of the geometric Littlewood--Richarson rule \cite[Theorems 5.6.8 and 5.6.9]{chhaibi2013littelmann}.


\subsection{Our objects of interest and contributions} 
As a preliminary step, we present Rider and Valk\'o's Theorem as a Matsumoto--Yor theorem in the setting developed by the first author, where matrices are replaced by block matrices. To us, this constitutes a first conceptual step toward defining a discrete-time version of the Rider--Valko's process that can be readily extended to Jordan algebras. Here the superscript $\flat$ indicates that we are working in the block setting. We start with an informal block dynamic written as $B^\flat=\left( B^\flat_t \ ; \ t \geq 0 \right)$, which is defined by 
\begin{align}\label{blockmod}
    d B^\flat_t & = B^\flat_t \circ
    \begin{pmatrix}
        db^\flat_t + \nu  \id_{M_r(\R)} dt & 0 \\
        dt  & -(db^\flat_t + \nu \id_{M_r(\R)} dt )^*
    \end{pmatrix},
    \ t\ge 0,
\end{align}
where $b^\flat = ( b^\flat_t \ ; \ t\ge 0)$ is the standard Brownian motion on $M_r(\R)$ and $\nu \in \R$.
More explicitly,
\begin{align}
    B^\flat_t & =
    \begin{pmatrix}
        g_t & 0 \\
        \left( 
            \int_0^t (g_s g_s^*)^{-1} ds
        \right) g_t
        & (g^*)^{-1}
    \end{pmatrix},
    \ t\ge 0,
\end{align} 
where $g = ( g_t \ ; \ t\ge 0)$ satisfies $d g_t=g^\flat_t\circ d(b_t+\nu \id_{M_r(\R)} t)$, i.e. $g$ is a drifted left Brownian motion on $GL_r(\R)$. Thus the Markov process $Z$ of Theorem \ref{thm:VR} appears in the lower left corner of $B^\flat$. In this block settings, the generalisations of the integral process $\iota$ and the Markov process $\ell$ are respectively
$$ \iota = \left( \iota_t := \int_0^t (g_s g_s^*)^{-1} ds \ ; \ t\ge 0 \right) $$
and $$ \ell = \left( \ell_t := g_t^* \left( \int_0^t (g_s g_s^*)^{-1} ds \right) g_t \ ; \ t\ge 0 \right)
  \ .
$$

\medskip

Upon replacing blocks by Jordan algebra elements, we have a unified point of view that highlights the essential ingredients. Only then could we identify natural Matsumoto--Yor and Rider--Valk\'o processes, both in the discrete and continuous time settings. Notice that in the discrete time setting, there are discrepancies with \cite{arista2024Matsumoto}.
For example, their Matsumoto--Yor process $S = \left( S(n) \ ; \ n \geq 0 \right)$ does not arise in our framework, as it is defined using random elements associated with both time steps $n$ and $n-1$.  
Our main statements are given in the discrete time setting in Subsection \ref{subsection:main_discrete} and in the continuous setting in Subsection \ref{subsection:main_continuous}.
There, we introduce Rider--Valk\'o processes, and we state the Matsumoto--Yor Markov property. 

Specializing our result to the continuous setting and for the cone of real symmetric matrices recovers \cite{rider2016matrix}. 
Specializing to the case of complex symmetric matrices, we also recover the matrix Dufresne identities in the physics literature \cite[Section 7]{grabsch2020wigner} and \cite{gautie2021matrix}. The latter references deal with the complex versions of the Lévy transform $\ell$ and associated $AX+B$ models, yet do not seem to treat the associated Matsumoto--Yor processes. Of course, we expect that the existing approaches \cite{rider2016matrix, arista2024Matsumoto} naturally adapt to the complex and quaternionic cases.
Specializing our result to the Lorentz cone or the exceptional octonion cone yields genuinely new results. The example of the Lorentz cone is specifically detailed in Subsection \ref{subsection:lorentz}.

Regarding the proof, we start by proving the result in the discrete setting, where distributional identities and the tool of intertwining are easier to manipulate. 
We believe that the discrete-time proofs presented in Section \ref{section:discrete_time} offer a particularly concise and conceptually optimal approach, even in the initial Matsumoto--Yor setting. Once the appropriate structures are identified, establishing intertwining in discrete-time reduces to analyzing a single Markov transition step and performing a change of variables. This conciseness reflects a deliberate effort to isolate the essential mechanisms prior to the Jordan algebra generalization.

Then, in Section \ref{section:continuous_time}, we deduce the continuous time result thanks to a functional central limit theorem on Lie groups. Here, the Lie group in question is the automorphism group of the Jordan algebra, which is nothing but matrix groups playing an analogous role to $GL_n(\R)$ in \cite{rider2016matrix}. A particularly interesting aspect is the interplay between Brownian motion on a Lie group $G$ and Brownian motion on the symmetric space $G/K$.

\section{Intuition and statement of the main theorems}
\label{intuition}

\subsection{Block setup for the discrete time Rider--Valk\'o's process} 
Our intuition stems from the block model in Eq. \eqref{blockmod}. Let us formally describe a block model in the discrete case. Let $(w_n; n\ge 0)$ be a sequence of i.i.d. random variables living in the set $S_r^+(\R)$. We consider a random process $G = (G_n \ ; \ n\ge 1)$ defined by
$$ G_{n+1}
   = \begin{pmatrix} W_{n+1} & 0 \\ Z_{n+1} & (W_{n+1}^*)^{-1} \end{pmatrix}
   = G_n \begin{pmatrix} w_{n+1} & 0 \\ \id_{GL_r(\R)} & w_{n+1}^{-1}\end{pmatrix}, 
   \ n\ge 0 ,$$
with $G_0 = \begin{pmatrix} \id_{GL_r(\R)} & 0 \\ 0 & \id_{GL_r(\R)} \end{pmatrix} $, i.e. 
$$ W_{n+1} = W_n w_{n+1} \, \textrm{ and }\, Z_{n+1}=Z_nw_{n+1}+(W_n^*)^{-1}, \, n\ge 0,$$ 
with $W_0 = \id_{GL_r(\R)}$ and $Z_0=0_{M_r(\R)}$. This model forces the appearance of two remarkable processes. On the one hand, we have an increasing process $I = ( I_n := Z_nW_n^{-1} \ ; \ n\ge 0)$ which satisfies
$$ I_n = Z_nW_n^{-1}
       = \sum_{k=0}^{n-1} w^{-1}_1\dots w^{-1}_kw^{-1}_{k+1}w^{-1}_k\dots w^{-1}_1, \quad n\ge 0$$ 
and on the other hand, we have an $AX+B$--type Markov chain $L = ( L_n := W_n^*Z_n  \ ; \ n\ge 0)$ which satisfies for $n \geq 0$
$$ L_{n+1} = W_{n+1}^*Z_{n+1}
           = w_{n+1}W_n^*Z_nw_{n+1} + w_{n+1}
           = w_{n+1} L_n w_{n+1} + w_{n+1} \ .$$


It is important to notice at this stage that the definitions of the two processes $I$ and $L$ make sense in the framework of any Euclidean Jordan algebras. Their definitions actually  involve  the so-called quadratic representation, which is written in this case
$\textrm P: S_r(\R) \rightarrow \mathrm{End}\left( S_r(\R) \right)$ and defined as
$$
    \textrm P(w)(s) := w s w, \ \ w,s \in S_r(\R) \ ,
$$
where $S_r(\R)$ is the set of symmetric $r\times r$ real matrices.
Using this quadratic representation, we can reformulate the previous equations as
$$ I_n = \sum_{k=0}^{n-1} \textrm P(w_1^{-1})\dots \textrm P(w_k^{-1})(w_{k+1}^{-1}) \ ,$$ 
and 
$$ L_{n+1} = \textrm P(w_{n+1})( L_n ) + w_{n+1} \ .$$
Both equations are well-defined within the broader framework of general Euclidean Jordan algebras. The above will in fact correspond to the particular case of the Jordan algebra $E=S_r(\R)$. Nevertheless the recursion formula $$Z_{n+1}=Z_nw_{n+1}+(W_n^*)^{-1}, \, n\ge 0,$$
does not, because there is in general no associative product in these algebras. The mathematicians familiar with Jordan algebras might recognize that a (matrix) square missing. As such, if we consider the more fundamental ``squared'' process $\Lambda = (\Lambda_n \ ; \ n\ge 0)$ defined by $\Lambda_n=Z_n^*Z_n$, $n\ge 0$, then $\Lambda$ satisfies 
$$ \Lambda_{n+1}=(w_{n+1}+L_n^{-1})\Lambda_n(w_{n+1}+L_n^{-1}) \ ,$$
and $\Lambda_n=L_n(W_n^*W_n)^{-1}L_n$, $n\ge 0$. 

These equations make sense in the framework of Euclidean Jordan algebras. These are the processes $L$ and $\Lambda$ that we study in this paper and their analogue in the framework of Euclidean Jordan algebras. We basically prove that $\Lambda$ is a Markov process and prove a Dufresne-type identity involving an inverse Wishart distribution on a symmetric cone. Close in spirit to \cite{chamayou1999additive}, this last distribution occurs as the stationary distribution of the Markov chain $L$.

 \subsection{Euclidean Jordan algebras, definitions} 
 \label{subsection:jordan_definitions}
 
 We first give the essential definitions to state our main theorems. Jordan algebras were introduced in 1933 by Pascual Jordan. More details about the historical and modern context are left as comments in Subsection \ref{subsection:history_jordan}. They provide a generalization of the space of real symmetric matrices. Each Euclidean Jordan algebra is associated with a symmetric cone. Analysis on specific families of symmetric cones arises in various contexts, such as multivariate statistics, mathematical physics or statistical mechanics. Harmonic analysis in these cones is particularly rich. Magically, a lot of explicit computations can be done which lead in particular to generalizations of special functions as gamma functions, Bessel functions or hypergeometric functions \cite{faraut1994analysis}. 
 
 A Euclidean Jordan algebra is a Euclidean space $(E, \langle \cdot ,\cdot\rangle)$ equipped with a bilinear mapping
$$\left\{
    \begin{array}{lll}
        E\times E & \to & E\\
        (x,y) & \mapsto &x\cdot y
    \end{array}
\right.$$ 
and a neutral element $e$ satisfying
\begin{enumerate}\label{axiom}
  \item $x\cdot y=y \cdot x$
  \item $x\cdot((x\cdot x)\cdot y)= (x\cdot x)\cdot(x\cdot y)$
  \item $x\cdot e=x$
  \item $\langle x,y\cdot z\rangle =\langle x\cdot y,z\rangle$
  \end{enumerate}  
  
If $E$ is the Cartesian product of two Euclidean Jordan algebras with positive dimension, $E$ is nonsimple. Otherwise, $E$ is said to be simple. We will see below a classification of Euclidean simple Jordan algebras. The basic and well-known example is the algebra $S_r(\R)$ of $r\times r$ symmetric real matrices equipped with the Jordan product defined for $a,b\in S_r(\R)$ by 
\begin{align}
    \label{eq:jordan_quasimultplication}
    a\cdot b = & \ \frac{1}{2}(ab+ba) \ .
\end{align}

To each Euclidean Jordan algebra $E$, we attach the set of Jordan squares 
$$ \bar E_+ = \left\{x\cdot x \ ; \ x \in E \right\}$$ 
and its interior $E_+$. Every element $x$ in $E_+$ is invertible in $E_+$. When $E=S_r(\R)$, $E_+$ is the familiar set of positive definite symmetric real matrices $S_R^+(\R)$. In general $E_+$ is a symmetric cone. To each Euclidean simple Jordan algebra $E$, we also attach the automorphism group $G(E_+)$ of $E_+$ which is defined by 
$$ G(E_+) = \left\{g\in GL(E) \ ; \ g(E_+) = E_+ \right\} \ . $$
 The quadratic representation $P: E\to \mbox{End}(E)$ of $E$ is defined by 
\begin{align}
\label{eq:def_P}
\mathrm P(x) = 2\textrm L(x)^2 - \textrm L(x\cdot x) \ , \quad x\in E \ ,
\end{align} 
where $\textrm L: E \rightarrow \mbox{End}(E)$ is the left action i.e. $\textrm L(x)(y) = x\cdot y$, for $x,y \in E$. 
In the case when $E=S_r(\R)$, $\textrm P(x)(y)=xyx$, for $x,y\in S_r(\R)$.

\medskip

\paragraph{ \bf Classification} Let us quickly give the reader an idea of the new setting that Jordan algebras provide. The details required for proofs are left to Section \ref{section:jordan_algebras}. Up to a linear isomorphism, there are only five kinds of Euclidean simple Jordan algebras as originally stated in \cite[Fundamental Theorem 2, p.64]{jordan1993algebraic_collected}, although we recommend \cite[Table p.97]{faraut1994analysis}. 

If $\mathbb K$ denotes either the set of real numbers $\R$, the set of complex ones $\C$, the set of quaternions $\mathbb H$ or the set of octonions $\mathbb{O} $, and $S_r(\mathbb K)$ the set of Hermitian matrices valued in $\mathbb K$, endowed with the scalar product defined by $\langle x,y\rangle = \Tr(xy^*)$, $x,y\in S_r(\mathbb K)$ and the Jordan product
$$x\cdot y= \frac{1}{2}(xy+yx), \, x,y\in  S_r(\mathbb K),$$  
where $xy$ is the ordinary product of matrices. Then $S_r(\R)$, $r\ge 1$, $S_r(\C)$, $S_r(\mathbb H)$, $r\ge 2$, and the exceptional $S_3(\mathbb O)$ gives the list of the four first kinds. 

The fifth kind is the Euclidean space $\R^n$ with $n \ge 3$ equipped with the Jordan product
\begin{align}
\label{JordanproductLorentz}
  x\cdot y = & \left( \sum_{i=0}^{n-1}x_iy_i, \ x_0y_1+y_0x_1, \ \dots, \ x_0y_{n-1}+y_0x_{n-1} \right) \ ,
\end{align}
for $x=(x_0,\dots,x_{n-1})$, $y=(y_0,\dots,y_{n-1})$.

\subsection{Main theorem and consequences in discrete setting} 
\label{subsection:main_discrete}
Let $E$ be a simple Jordan algebra, $E_+$ the interior of the set of Jordan squares, $G(E_+)$ the automorphism group and $G$ the connected component of $G(E_+)$ containing the identity. 
One considers a $G$ invariant measure $\mu$ on $E_+$ which will be defined explicitly below and two families of laws on $E_+$ which are particularly relevant to our setting. Up to a scalar multiple, there is a unique scalar product satisfying the fourth axiom of Jordan algebra. We choose one which is often considered as canonical which will be defined later. 
The matrix GIG distribution $\mathrm{GIG}_{E_+}\left( p ; a, b \right)$ for $a,b\in E_+$, $p\in \R$ is a distribution on $E_+$ defined by 
\begin{align}
\label{eq:def_GIG} 
   & \ \mathrm{GIG}_{E_+}\left( p ; a, b \right)(dx) \\
:= & \ \frac{\det(x)^{p}}{2K(p;a,b)}
\exp\left( -\frac{1}{2}(\langle a,x\rangle +\langle b,x^{-1}\rangle) \right)
\mu(dx) \ ,
\nonumber
\end{align}
 where $\det$, defined below, coincides with the usual determinant when $E_+=S^+_r(\R)$. The Wishart distribution $\gamma_{p,a}$ for $a\in E_+$, $p>\dim E/r-1$, where $r$ is the rank which will be defined below, is the distribution on $E_+$ defined by 
\begin{align}
\label{eq:def_Wishart}
\gamma_{p,a}(dx) 
   :=\frac{2^{-rp}}{\det(a)^p\Gamma_r(p)} \det(x)^{p}e^{-\frac{1}{2}\langle a^{-1},x\rangle}
    \, \mu(dx),
   \ 
\end{align}
where $\Gamma_r(p)$ is the multivariate gamma function. The inverse Wishart distribution $\mathrm{Inv}_{\#}(\gamma_{p,a})$ is the image measure of $\gamma_{p,a}$ under the inverse map $\mathrm{Inv}: x\mapsto x^{-1}$.

\begin{definition}
\label{maindefi} 
Let $(w_n)_{n\ge 1}$ be a sequence of i.i.d. random variables on $E_+$ distributed according to $\mathrm{GIG}_{E_+}\left( p ; e, e \right)$. 

Let on the one hand, $L := (L_n \ ; \ n\ge 0)$ be the Markov chain defined by 
$$L_{n+1}=\mathrm{P}(w_{n+1})(L_n)+w_{n+1}, \, n\ge 0,$$
and $L_0=0$. 

On the other hand, the random process $\Lambda := (\Lambda_n \ ; \ n\ge 1)$ is defined by
$$ \Lambda_{n+1}=\mathrm{P}(w_{n+1}+L_n^{-1})(\Lambda_{n}), \, n\ge 1, $$
and $\Lambda_1=e$.
\end{definition}

Notice that the two-component random process $\left( (\Lambda_n,L_n) \ ; \ n\ge 0 \right)$ is clearly a Markov chain. Analogously to the historical Pitman 2M-X Theorem \ref{thm:pitman} and then the Matsumoto--Yor Theorem \ref{thm:matsutomo_yor}, our main theorem asserts that the process $\Lambda$ enjoys a remarkable Markov property together with a certain interplay with (the Jordan algebra version of) the GIG distribution.

\begin{thm}[Main Theorem, discrete time version]
\ 

\begin{enumerate}
    \item The random process $\left( \Lambda_n \ ; \ n\ge 1 \right)$ is Markovian.
    \item For any $n\ge 2$, the conditional law of $L_n$ given $\left( \Lambda_k \ ; \ k\le n\right)$ is $\mathrm{GIG}_{E_+}\left( p ; a^{-1}, e \right)$, where $a=\Lambda_n$.
\end{enumerate}
\label{thm:main_discrete}
\end{thm}

Following the historical point of view of Chris Rogers and Jim Pitman \cite{rogers1981markov}, the proof hinges on an intertwining relation between the Markov transition kernels of the latter Markov chain and of $\Lambda$.

\begin{definition}
\label{defiint} 
Let $I=(I_n \ ; \ n\ge 0)$ be the random process defined by 
$$ I_n=\sum_{k=0}^{n-1}\mathrm P(w_1^{-1})\dots \mathrm P(w_k^{-1})(w_{k+1}^{-1}) \ , \ n\ge 0 \ .
$$
\end{definition} 
The following theorem gives a discrete Dufresne-type identity. The key ingredient of the proof is to rely the limit of the series to the stationary distribution of the $AX+B$ Markov chain $L$, according to a principle for which we classically refer to Harry Kesten \cite{kesten1973random}, Wim Vervaat \cite{vervaat1979on}, Charles Goldie \cite{goldie1991implicit}, or Philippe Bougerol and Nico Picard \cite{bougerol1992strict}. One can also see Jean-Fran\c cois Chamayou and G\'erard Letac in a framework close to ours \cite{chamayou1999additive}, or \cite{buraczewski2016Stochastic} for a large overview of the subject.

\begin{thm}[Dufresne Identity for Jordan algebras, discrete version]
\label{thm:dufresne_discrete}
Suppose that $p>\dim E/r-1$. Then
$$ I_\infty := \lim_{n \rightarrow \infty} I_n
            = \sum_{k=0}^\infty \mathrm P(w^{-1}_1)\dots \mathrm P(w^{-1}_{k})(w^{-1}_{k+1})
$$ 
is almost surely finite and distributed according to an inverse Wishart distribution $\mathrm{Inv}_{\#}(\gamma_{p,e})$.
\end{thm} 

\subsection{Continuous time theorems} 
\label{subsection:main_continuous}
 
In this Subsection, we give continuous time versions of the Main Theorem \ref{thm:main_discrete} and the discrete Dufresne identity given as Theorem \ref{thm:dufresne_discrete}. These are of course obtained by carefully renormalizing the previous results. Moreover, the choice of such a presentation is justified by the fact that it is easier to obtain the Matsumoto--Yor Markov property in the discrete context first.

More specifically, as $G$ is a connected Lie group, one needs to invoke Donsker-type theorems on Lie groups \cite{stroock1973limit}. The Lie algebra $\gfrak$ is a subalgebra of $\mathfrak{gl}(E)$ which contains $\id_{\mathfrak{gl}(E)} = \Lrm(e)$. It is equipped with an $\Ad(K)$-invariant inner product. One considers a left Brownian motion on $\left( g_t \ ; \ t\ge 0 \right)$ on $G$ with covariance $4 \id_{\mathfrak{g}^{\otimes 2}} $ and drift $2p \id_{\mathfrak{gl}(E)} = 2p \Lrm(e)$, that is to say
\begin{align}
    \label{eq:def_g_Jordan}
    dg_t & = g_t \circ \left( 
      \sum_{i=1}^{\dim G} 2 X_i dB^{(i)}_t
    + 2p \Lrm(e) dt
    \right) \ ,
\end{align}
where $\left( X_1, \dots, X_{\dim G} \right)$ is an orthonormal basis of $\gfrak \subset \mathfrak{gl}(E)$, and the $B^{(i)}$ are standard real Brownian motions. The canonical scalar product is defined later in Section \ref{section:jordan_algebras}.

One also defines the random processes $\iota = (\iota_t \ ; \ t\ge 0)$, $\ell = (\ell_t \ ; \ t\ge 0)$ and $\lambda = (\lambda_t \ ; \ t\ge 0)$ by 
\begin{align}
\label{threeproc}
\iota_t   = \int_0^t(g_s^*)^{-1}(e) ds \ , \quad
\ell_t    = g_t^*(\iota_t), \textrm{ and } \, 
\lambda_t = g_t^*(\iota_t\cdot \iota_t ), \quad t\ge 0 \ .
\end{align} 

\begin{thm}[Main Theorem, continuous time version]
\label{thm:main_continuous} \ 

The random process $\lambda$ is a Markov process, whose natural filtration is strictly contained in that of $g$, i.e. for all $t > 0$, $\Fc_t^\lambda \varsubsetneq \Fc_t^g$.

Furthermore, the conditional law of $\ell_t$ given $\left( \lambda_s \ ; \ s \le t \right)$ is the distribution $\mathrm{GIG}_{E_+}\left( p ; a^{-1}, e \right)$, where $a=\lambda_t$ i.e.
\begin{align*}
     \ \P\left( \ell_t \in dx \ | \ \Fc_t^\lambda, \ \lambda_t = \lambda \right)
 = & \ \mathrm{GIG}_{E_+}\left( p ; \lambda^{-1}, e \right)(dx) \ .
\end{align*}
\end{thm} 
To get the theorem, we exploit as in \cite{matsumoto2000analogue} an intertwining relation between $(\lambda,\ell)$ and $\lambda$ which is carried out from the discrete analogous processes. Notice that the conditional laws involved in Theorem \ref{thm:main_continuous} are the same as the conditional laws involved in their discrete time counterpart.
This is a key ingredient in the proofs of these theorems, and it hints to the integrability of the constructions at hand. In the setup of Matsumoto--Yor and its Lie-theoretic generalizations in \cite{chhaibi2013littelmann}, we are dealing with the harmonic analysis of Whittaker functions and the related geometric crystals. This aspect was further analyzed in the recent paper \cite{herent2024discrete}.

Similarly, to obtain the continuous time version of the theorems we will defined approximations of the diffusion processes involved in these theorems that all satisfy the same conditional property. Proposition \ref{first-step}, which establishes a remarkable property of the GIG distributions, allows to construct such approximations.

\begin{thm}[Dufresne Identity for Jordan algebras, continuous version]
\label{thm:dufresne_continuous}
Suppose that $p>\dim E/r-1$. Then the perpetuity integral 
 $$ \iota_\infty 
  := \lim_{t \rightarrow \infty} \iota_t
  = \int_0^\infty (g_s^*)^{-1}(e) ds $$ 
is almost surely finite and distributed according to an inverse Wishart distribution $\mathrm{Inv}_{\#}( \gamma_{p,e} )$.
\end{thm}

\section{Further comments and examples}

\subsection{Considerations on Brownian motion on Lie groups}
Of course, the diffusive limit of random walks always makes sense and gives natural candidates for Brownian motion on $G$. But a few remarks are in order. Here, we argue that, while Brownian motion on the maximal compact subgroup $K$ and the symmetric space $G/K$ are natural, {\it the} Brownian motion on a non-compact Lie group $G$ is {\it not} as natural from the point view of geometry and harmonic analysis.

To being with, what do we usually mean by {\it the} Brownian motion on $G$? Following for example \cite{Taylor1991brownian}, one picks an infinitesimal generator which is a half sum of squares, $\half \Delta_G = \half \sum_{i=1}^{\dim G} X_i^2$, where the $X_i$'s are orthogonal w.r.t. the positive--definite quadratic form $\kappa_\theta$ defined from the Killing form $\kappa$ and the Cartan involution $\theta$. More explicitly and up to normalization, it is given for $(X,Y) \in \gfrak \times \gfrak$ by $\kappa_\theta(X,Y) = -\kappa(X, \theta Y)$. In the case of  $G = SL_{r}(\R)$, $\gfrak = \mathfrak{sl}_{r}(\R)$ is the space of zero--trace matrices, $\kappa(X,Y) = 2n \tr(X Y)$ and $\theta X = -X^*$ so that $\kappa_\theta(X,Y) = 2n \tr(X Y^*)$ is indeed the usual positive--definite scalar product on matrices. Using the Stratonovich convention, the stochastic differential equation defining the (left-invariant) Brownian motion on $G$ is thus
\begin{align}
\label{eq:def_BM_G}
    dg_t = & \ g_t \circ \left( \sum_{i=1}^{\dim G} X_i dB_t^{(i)} \right) \ ,
\end{align}
where the $B^{(i)}$'s are independent standard real Brownian motion.

Here it is well known that the Killing form $\kappa$ is, up to scaling, the unique invariant form. However it is not positive nor negative definite on $\gfrak$, making the twisting by $\theta$ absolutely necessary. In the decomposition $\gfrak = \pfrak \oplus \kfrak$, $\kappa$ is negative definite on $\kfrak$ and positive definite on $\pfrak$.
They are respectively the tangent spaces to the compact form $K$ and the symmetric space $G/K$, for which no twisting is necessary.

\medskip

{\bf On the compact form $K$:} On the compact group, the Laplacian is nothing but the sum of squares operator, along vectors orthonormal w.r.t. the (untwisted) $-\kappa$. In the absence of twisting, it truly coincides with the Casimir element, which plays a crucial role in the representation theory, and generates a natural Brownian motion. By picking the orthonormal basis of $X_i$'s as a basis adapted to the decomposition $\gfrak = \pfrak \oplus \kfrak$, the Brownian motion on $K$ takes a similar form to Eq. \eqref{eq:def_BM_G}, that is to say
\begin{align}
\label{eq:def_BM_K}
    dk_t = & \ k_t \circ \left( \sum_{i=1+\dim P}^{\dim K + \dim P} X_i dB_t^{(i)} \right) \ .
\end{align}

\medskip

{\bf On the symmetric space $G/K$:} Thanks to the polar decomposition, $P \approx G/K$ and a natural Brownian motion can be defined on $G/K$. Nevertheless notice that $P$ is not a group and equivalently $\pfrak$ is not a Lie algebra due to the classical relation $[\pfrak, \pfrak] = \kfrak$. In fact, since the tangent space to $G/K$ is naturally identified by translation to $\left( \gfrak \mod \kfrak \right)$, one can choose any supplementary space $\mathfrak{s}$ to $\kfrak$. In order to construct {\it any} diffusion, including the natural Brownian motion, one can consider infinitesimal increments supported on $\mathfrak{s}$ or equivalently a second order differential operator on $\mathfrak{s}$. This yields a subelliptic process on $G$ whose law in $G/K$ is the desired diffusion. Let us give two choices of $\mathfrak{s}$, which yield different descriptions of the Brownian motion on $G/K$. 

The direct decomposition provided by the Iwasawa decomposition $\gfrak = \nfrak \oplus \afrak \oplus \kfrak$ is particular as it yields $\mathfrak{s} = \nfrak \oplus \afrak$, which is a Lie algebra. Infinitesimal increments in $\mathfrak{s} = \nfrak \oplus \afrak$ yield a process which never leaves the group $NA$. This yields a realization of the Brownian motion on $G/K$ by an invariant process in $NA$ -- clearly not hypoelliptic in $G$. In fact, this process has a generator which is a sum of squares plus a drift term. This realization of the natural Brownian motion on $G/K$ as an invariant process on the solvable group $NA$ provides many interesting developments, see for example \cite[Deuxième Partie, Définition]{bougerol1983exemples} and \cite{bougerol2002paths}.

Using increments in $\mathfrak{s} = \pfrak$, one obtains a process which does not remain in $P$, and is hypoelliptic in $G$ -- again thanks to the relation $[\pfrak, \pfrak] = \kfrak$. This choice is rather exotic and, to the best of our knowledge, has only appeared in the paper of Marie-Paule and Paul Malliavin \cite{malliavin74}. For reasons specific to Jordan algebras, we shall make the same choice. The associated left-invariant hypoelliptic Brownian motion has a stochastic differential equation of the form
\begin{align}
\label{eq:def_BM_G_hypo}
    d\widetilde{g}_t = & \ \widetilde{g}_t \circ \left( \sum_{i=1}^{\dim P} X_i dB_t^{(i)} \right) \ .
\end{align}

\medskip

In the end, the previous discussion shows that in $G/K$, the natural Brownian motion can appear in multiple forms. In this paper, due to the importance of the quadratic representation in Jordan algebras, which lives in $P$, we make the choice of using the Malliavin--Malliavin \cite{malliavin74} hypoelliptic Brownian motion with noise in $\pfrak$. In the spirit of the previous discussion, we now state and prove an identity in law relating {\it the} Brownian motion $\left( g_t \ ; \ t \geq 0 \right)$ on $G$ which appears in our Main Theorem \ref{thm:main_continuous}, and the Malliavin--Malliavin hypoelliptic Brownian motion $\left( \widetilde{g}_t \ ; \ t \geq 0 \right)$ which will appear in our proofs.

\begin{lemma}
\label{lemma:factorization_gk}
Let $k = (k_t \ ; \ t  \geq 0)$ be a left-invariant Brownian motion on $K$ independent of $\widetilde{g}$.
Then the process $\widetilde{g} k = (\widetilde{g}_t k_t \ ; \ t \geq 0)$ has the same distribution as the Brownian motion $g = \left( g_t \ ; \ t \geq 0 \right)$ on $G$. 
\end{lemma}
\begin{proof} We start by invoking the Leibniz rule for semi-martingales (or Itô calculus) and the definition of $\widetilde g$ in Eq. \eqref{eq:def_BM_G_hypo}. Here, we use in the Stratonovich convention.
 \begin{align*}
     d( \widetilde{g}_t k_t )
     & = ( d\widetilde{g}_t)  \circ k_t  + 
     \left( \widetilde{g}_t k_t \right) \circ (k_t^{-1} dk_t) \\
     & = \widetilde{g}_t \circ ( \sum_{i=1}^{\dim P} X_i dB_t^{(i)}  ) \circ k_t + 
     \left( \widetilde{g}_t k_t \right) \circ (k_t^{-1} dk_t)\\
     & = \widetilde{g}_t k_t \circ \mathrm{Ad}_{k_t^{-1}}( \sum_{i=1}^{\dim P} X_i dB_t^{(i)} ) + 
     \left( \widetilde{g}_t k_t \right) \circ (k_t^{-1} dk_t)\\
     & = ( \widetilde{g}_t k_t ) \circ \left[ \mathrm{Ad}_{k_t^{-1}}( \sum_{i=1}^{\dim P} X_i dB_t^{(i)} ) + 
     k_t^{-1} dk_t \right]\ .
 \end{align*}
 Now, because of Lévy's characterization of Brownian motion, and the $\Ad(K)$-invariance of the covariance structure, the increments $$\Ad_{k_t^{-1}}( \sum_{i=1}^{\dim P} X_i dB_t^{(i)} )$$ have the same distribution as $\sum_{i=1}^{\dim P} X_i dB_t^{(i)}$. Thus we find the same stochastic differential equation as Eq. \eqref{eq:def_BM_G}. In law, $\tilde g k$ is indeed {\it the} left-invariant Brownian motion on $G$.
\end{proof}

\subsection{Historical comments on Jordan algebras}
\label{subsection:history_jordan}
In his seminal work \cite{jordan1933}, Pascual Jordan sets the grounds for what is now referred to as Jordan algebras. The starting point is following simple considerations from quantum mechanics. Only for the purpose of this Subsection, consider $\Ac$ an algebra of quantum observables. Following the principles of quantum mechanics, only self-adjoint operators correspond to measurements according to the probabilistic Born rule. Given two operators $A, B \in \Ac$, the operator $A+B$ corresponds to natural measurement. However the product $A \times B$ does not -- it is not even necessarily self-adjoint. By defining {\it quasimultiplication} as the bilinear operation $A \cdot B = \half \left( AB + BA \right)$, which maps self-adjoint operators to self-adjoint operators, Jordan realizes that multiplication tables of angular momenta and certain equations from quantum mechanics have simpler forms. This operation being exactly Eq. \eqref{eq:jordan_quasimultplication}, and being non-associative, the genesis of Jordan algebras becomes clear.

In the next paper \cite{jordan1993algebraic}, in order to fully understand what the new formalism brings, Jordan joins John von Neumann and Eugene Wigner in order to algebraically classify Jordan algebras, leading to the list given in Subsection \ref{subsection:jordan_definitions}. Their goal was to get an improved quantum theory on the basis of algebras with non-associative multiplication. According to some physicists, including Paul Dirac \cite{dirac1940}, this ``attempt was not sucessful'' for rather aesthetic reasons. In our opinion, perhaps the most disappointing aspect is that the classification does not yield completely new objects. Nevertheless, this was the starting point of multiple constructions of non-associative algebras \cite{raffin1951anneaux}. Finally let us mention the link between the exotic octonion Jordan algebra and the standard model in physics according to the more recent \cite{todorov2018octonions, boyle2020}.

In any case, our motivation remains squarely grounded in Pitman-type theorems, and here, the Matsumoto--Yor property.

\subsection{The example of the Lorentz cone}
\label{subsection:lorentz}
Here we give more details about Theorems \ref{thm:main_continuous} and \ref{thm:dufresne_continuous} when specializing to the Jordan algebra $E$ of the fifth kind. Let $n \ge 3$. Let us begin to describe the various objects attached to the Jordan algebra $E=\R^n$, with a Jordan product defined by Eq. \eqref{JordanproductLorentz} and equipped with the usual inner product. The complete description is given in \cite[I.2. Examples]{faraut1994analysis}. In this case, $E_+$ is the Lorentz cone 
$$E_+= \{x\in E \ ; \ x_0>0, \, (x,x) >0\},$$
where $(\cdot,\cdot)$ is the Lorentz quadratic form given for $x,y\in \R^n$ by
$$ (x,y) := x_0y_0-x_1y_1-\dots -x_{n-1}y_{n-1} \ .$$
Let $SO_0(1,n-1)$ be the identity component of the subgroup of elements of $GL(\R^n)$ preserving the Lorentz quadratic form $(\cdot,\cdot)$. We have that $G = \R_+^* \times SO_0(1,n-1)$, i.e. $G$ is the direct product of the group of positive dilations with $SO_0(1,n-1)$. Notice that dilations are central. The subgroup $K$ defined by 
$$ K = \left\{
\begin{pmatrix}
  1  & 0 \\
    0 & u
\end{pmatrix} \ ; \ u\in SO(n-1) \right\} $$
is a maximal compact subgroup of $G$. 
If $(e_0,\dots,e_{n-1}) $ is the canonical basis of $\R^n$ then $e=e_0$ and 
$K=\{g\in G \ ; \ g(e)=e\}$.  The Lorentz cone is identified to $G/K$ via 
$$ g\in G/K\mapsto g(e)\in E_+ \ .$$
The hyperbolic space $\H^{n-1}$ of dimension $n-1$ defined by 
$$ \H^{n-1} = \left\{ x\in \R^n \ ; \ (x,x)=1, \ x_0 > 0 \right\} $$
is identified to $SO_0(1,n-1)/K$ via 
$$ g\in SO_0(1,n-1)/K\mapsto g(e)\in \H^{n-1} \ .$$
Each $x\in \H^{n-1}$ can be written uniquely as 
$$x=(\cosh r)e_0+(\sinh r)\varphi,$$
with $(r,\varphi)\in \R_+\times \S^{n-2}$ where $\S^{n-2} = \left\{\sum_{k=1}^{n-1}\varphi_k e_k\ ;\ \sum_{k=1}^{n-1} \varphi_k^2 = 1 \right\}$. The coordinates $(r,\varphi)$ are called the polar coordinates of $x$. The Laplace--Beltrami operator on $\H^{n-1}$ is given in these coordinates by 
$$ \half \Delta_{\H^{n-1}} = 
   \half \frac{\partial^2}{\partial r^2} 
   + \frac{n-2}{2}\coth r\frac{\partial}{\partial r} + 
   \frac{1}{\sinh^2 r} \left( \half \Delta_{\S^{n-2}}^\varphi \right) \ , 
$$ 
where $\Delta_{\S^{n-2}}^\varphi$ is the Laplace operator on the sphere $\S^{n-2}$ \cite[Theorem III.5.3]{franchi2012hyperbolic}.

Let now $(g_t \ ; \ t\ge 0)$ be as in Eq. \eqref{eq:def_g_Jordan}. Let $(b_t\, ; \ t\ge 0)$ and $(\xi_t \ ; \ t \ge 0)$ be two processes with values respectively in $\R$ and $\H^{n-1}$ such that 
$$ (g^*_t)^{-1}(e) = e^{-2b_t-2pt} \xi_{4t} = e^{-2b_t^{(p)}} \xi_{4t} \ , 
$$ 
for $t \ge 0$. Clearly, because of the direct product structure and the fact that the drift $p L(e)$ is central, we see that $e^{-2b}$ and $\xi$ are independent Brownian motions on respectively $\R_+^*$ and $SO_0(1,n-1)/K \approx \H^{n-1}$. In particular $b$ is a standard real Brownian motion. Now, following \cite[Theorem III.5.2]{franchi2012hyperbolic}, the Casimir operator on $SO_0(1,n-1)$ associated to the Lorentz quadratic form and the generator of $g$ (without the central part) agree on functions on $SO_0(1,n-1)$ which are invariant under the right
action of $K$, up to normalization. From that, we can deduce that $(\xi_t \ ; \ t\ge 0)$ is a Brownian motion on $\H^{n-1}$ with generator $\half \Delta_{\H^{n-1}}$. 

A more explicit expression would be to write
$$
    \xi_t = \left( \cosh R_t \right) e_0 + \left( \sinh R_t \right) \Phi_{\int_0^t \frac{ds}{\sinh^2 R_s}} \ ,
$$
where $\Phi$ is the angular process and $R$ is the hyperbolic radial process. The process $\Phi$ is the Brownian motion on $\S^{n-2}$. The process $R$, which is also known as the hyperbolic Dyson Brownian motion, has generator $\Lc^R = \half \frac{\partial^2}{\partial r^2} + \frac{n-1}{2}\coth r\frac{\partial}{\partial r}$. 

In that case, the processes $\left( \iota, \lambda \right)$ appearing in Subsection  \ref{subsection:main_continuous} are given for $t \geq 0$ by
$$ \iota_t = \int_0^t e^{-2b_s^{(p)}} \xi_{4s} ds 
           = \int_0^t e^{-2b_s^{(p)}} \left( \left( \cosh R_{4t} \right) e_0 + \left( \sinh R_{4t} \right) \Phi_{\int_0^{4t} \frac{ds}{\sinh^2 R_s}} \right) ds
           \ ,$$
and 
$$ \lambda_t = g_t^*\left( \int_0^t e^{-2b_s^{(p)}} \xi_{4s} ds
                     \cdot \int_0^t e^{-2b_s^{(p)}} \xi_{4s} ds \right) \ .$$
Our Matsumoto--Yor property from Theorem \ref{thm:main_continuous} states that $\lambda$ is Markovian in its own filtration. And our Dufresne identity from Theorem \ref{thm:dufresne_continuous} states that if $p>\frac{n}{2}-1$ then 
$$ \iota_\infty = \int_0^\infty e^{-2b_s^{(p)}} \xi_{4s} \ ds 
$$ 
is distributed according to an inverse Wishart distribution $\mathrm{Inv}_{\#}( \gamma_{p,e} )$ on the Lorentz cone.

\begin{rmk}
The usual Dufresne identity is morally recovered when $n=1$, so that the hyperbolic space degenerates to a point and $\xi \equiv 1$. The inverse Wishart distribution indeed degenerates to an inverse gamma.
\end{rmk}

\begin{rmk}
Finally let us mention that the octonion case is even more exotic. And it would be interesting to investigate that case further, as perhaps, the Matsumoto--Yor property reflects something special about this algebra.
\end{rmk}

\section{Euclidean Jordan algebras}
\label{section:jordan_algebras}
 
In this section we recall the essential definitions and first properties that we need about the Euclidean Jordan algebras as it is done in \cite{casalis1996lukacs}.
They all can be found in the book of Jacques Faraut and Adam Kor\'anyi \cite{faraut1994analysis}. Let $E$ be a simple Euclidean Jordan algebra.

\medskip

\paragraph{\bf Rank} An element $c$ of $E$ is said to be an  idempotent if $c \cdot c=c$. Two idempotents $c$ and $d$ are said to be orthogonal if $c\cdot d=0$. In that case, $c$ and $d$ are also orthogonal with respect to the scalar product. We say that an idempotent is primitive if it is not zero and cannot be written as the sum of two non-zero  idempotents. One says that $c_1, \dots , c_r$ is a complete system of primitive orthogonal idempotents if each $c_k$ is a primitive idempotent and
$$\sum_{i=1}^rc_i=e \textrm{ and } c_i \cdot c_j=\delta_{ij}c_{i} \, \textrm{ for } 1 \le i,j\le r \ .$$ 
The size $r$ of such a system depends only on $E$ and is called the rank of $E$. If $E$ is $S_r(\mathbb K)$ then the rank is $r$. If $E$ belongs to the fifth kind of simple Euclidean Jordan algebras, then $r=2$.

\medskip

\paragraph{\bf Spectral decomposition, determinant and trace} Any element $x$ of $E$ can be written $x=\sum_{i=1}^r\lambda_ic_i$ in a complete system of primitive orthogonal idempotents, where $\lambda_1,\dots, \lambda_r $ are real numbers called the eigenvalues of $x$ \cite[Theorem III 1.2]{faraut1994analysis}. One defines the trace of $x$ by $\Tr(x)=\sum_{i=1}^r\lambda_i$ and its determinant by $\det(x)=\prod_{i=1}^r\lambda_i$. Note that the bilinear map $(x,y) \in E \times E \mapsto \Tr(x\cdot y)$ defines a scalar product on $E$ satisfying the axioms of the Euclidean Jordan algebra. It is often chosen as the canonical scalar product.

\medskip

\paragraph{\bf Symmetric cones} One considers the set of Jordan square $\bar E_+ $ previously defined and its interior $E_+$. Note that if $x\in E$ has the spectral decomposition $x=\sum_{i=1}^r\lambda_ic_i$ then its square is given by $x\cdot x=\sum_{i=1}^r\lambda_i^2c_i$. Thus $E_+$ represents the subset of elements in $E$ with strictly positive eigenvalues.
 
\medskip

\paragraph{\bf Quadratic representation} The quadratic representation $\mathrm P: E\to \mbox{End}(E)$, as defined earlier, is self-adjoint, that is to say that for $x,y,z\in E$, it satisfies $\langle x,\mathrm P(y)(z)\rangle=\langle \mathrm P(y)(x),z\rangle$. Moreover, for $x\in E$, $\mathrm P(x)(e)=x\cdot x$.

\medskip

The following properties are crucial for the proofs. We recall them even when they are quite elementary. Let $E^{\times} \subset E$ be the subset of invertible elements in $E$ for the Jordan product.

\begin{proposition}[\cite{faraut1994analysis}, Proposition II 3.1, 3.3]
\label{propF1} $ $
\begin{enumerate}
\item An element $x$ in $E$ is invertible if and only if the quadratic representation $\mathrm P(x)$ is invertible. If $x$ in $E$ is invertible then $\mathrm P(x)^{-1}=\mathrm P(x^{-1})$.
\item If $x$ is invertible then $\mathrm P(x)(x^{-1})=x$.
\item If $x,y\in E$ then $\mathrm P(\mathrm P(y)(x))=\mathrm P(y)\mathrm P(x)\mathrm P(y)$.
\item If $x,y\in E$ are invertible then $(\mathrm P(x)(y))^{-1}=\mathrm P(x^{-1})(y^{-1})$.
\item The differential of $x\in E^{\times} \mapsto x^{-1}$ at $x$ is $-\mathrm P(x^{-1})$.
\end{enumerate} 
\end{proposition}
Notice that, as the operators $\mathrm L(x):y\in E\mapsto x\cdot y$ and $\mathrm P(x)$ commute, the first point of the proposition implies in particular that 
$\Prm(x^{-1})(x\cdot x) = \Prm(x^{-1}) P(x) (e) = e$, $x\in E$.

\begin{proposition}[\cite{faraut1994analysis} Proposition III 2.2] 
 If $x\in E_+$ then $P(x)\in G$. In particular if $x,y\in E_+$ then $\mathrm P(x)(y)\in E_+$.
\end{proposition}
 
\begin{proposition}[\cite{faraut1994analysis}, Proposition III 4.2]\label{III4-2}
 Suppose that $E$ is a simple Euclidean Jordan algebra. For $x\in E$
 \begin{align*}
 \det(\mathrm P(x))&=\det(x)^{2\dim E/r} \ ,\\
 \det(\mathrm P(y)(x))&=\det(y)^2\det(x) \ .
 \end{align*}
\end{proposition}
 
\begin{proposition}[\cite{faraut1994analysis}, Lemma X 4.4]
\label{proposition:forlim} 
If $a,b\in E_+$ then 
$$ \mathrm P(a^{-1}+b^{-1})=\mathrm P(a^{-1})\mathrm P(a+b)\mathrm P(b^{-1}) \ .
$$ 
\end{proposition}

We state \cite[Exercice 5.c, Chapter II]{faraut1994analysis} as a proposition.
\begin{proposition}
\label{magic-identity} 
If $a,b\in E_+$ then 
$$(a+\mathrm P(a)(b^{-1}))^{-1}+(a+b)^{-1}=a^{-1}.$$ 
\end{proposition}

\section{Proofs in the discrete time setting}
\label{section:discrete_time}

The measure $\mu$ on $E_+$ is defined by
\begin{align}
  \label{eq:def_mu}
  \mu(dx) & = \ \vert \det(x)\vert^{-\frac{\dim E}{r}} \, dx,
\end{align}
where $dx$ is the Lebesgue measure on $E_+$. Proposition III.4.3 in \cite{faraut1994analysis} implies that the measure $\mu$ is $G$-invariant. It is also invariant under the inverse map $\mathrm{Inv}: x\mapsto x^{-1}$. The matrix GIG distribution $\mathrm{GIG}_{E_+}\left( p ; a, b \right)$ for $a,b\in E_+$, $p\in \R$ is and the Wishart distribution $\gamma_{p,a}$ for $a\in E_+$, $p>\dim E/r-1$ are respectively defined by Equations \eqref{eq:def_GIG} and \eqref{eq:def_Wishart}. 

\subsection{Proof of Theorem \ref{thm:main_discrete}: The Matsumoto--Yor property}
The proof of Theorem \ref{thm:main_discrete} rests on Proposition \ref{first-step} for which we need the following lemma.
\begin{lemma}
\label{cov}
Consider the transformation $(x,y) \in E_+ \times E_+ \mapsto (u,v)$ where 
$$ u = \mathrm P(y)(x)+y \ , \quad v=y+x^{-1} \ .$$
If $J(x,y)$ is the Jacobian of this change of variables, then for $x,y\in E_+$, 
$$ \vert\det(J(x,y))\vert=\vert\det(\mathrm P(y+x^{-1}))\vert \ .$$
\end{lemma}
\begin{proof} From the identity of Proposition \ref{magic-identity}, $ u=(y^{-1}-(y+x^{-1})^{-1})^{-1}$. Using the properties of Proposition \ref{propF1} and properties of determinant of block matrices one obtains
\begin{align*} 
\vert \det(J(x,y))\vert &=\vert\det\begin{pmatrix}
   \mathrm P(u)\mathrm P(v^{-1})\mathrm P(x^{-1})  & \mathrm P(u)(\mathrm P(y^{-1})-\mathrm P(v^{-1})) \\
    -\mathrm P(x^{-1}) & \id_{G}
\end{pmatrix}\vert\\
&=\vert\det(\mathrm P(u)\mathrm P(y^{-1})\mathrm P(x^{-1}))\vert\\
&=\vert\det(\mathrm P(\mathrm P(y)x+y)\mathrm P(y^{-1})\mathrm P(x^{-1}))\vert\\
&=\vert\det(\mathrm P(y)\mathrm P(x+y^{-1})\mathrm P(x^{-1}))\vert \ .
\end{align*}
The expected identity follows from Proposition \ref{proposition:forlim}.
\end{proof}
\begin{proposition}\label{first-step} Let $a,b,\nu_1,\nu_2\in E_+$. Let $X$ and $w$ be two independent random variables on $E_+$ such that  the law of $X$ and $w$ are respectively $\mathrm{GIG}_{E_+}\left( p ; a^{-1}, \nu_1 \right)$, and $\mathrm{GIG}_{E_+}\left( p ; \nu_1, \nu_2 \right)$.
Then the law of $\mathrm P(w)(X)+w$ given $\mathrm P(w+X^{-1})(a)=b$
is $\mathrm{GIG}_{E_+}\left( p ; b^{-1}, \nu_2 \right)$.
\end{proposition} 
\begin{proof} Let $U=\mathrm P(w)(X)+w$ and $V=w+X^{-1}$. Using the change of variable of lemma \ref{cov}, i.e. $y=(u^{-1}+v^{-1})^{-1}$, $x=\mathrm P(v^{-1})(u)+v^{-1}$, and Proposition \ref{magic-identity}, one obtains easily 
that the density $f_{U,V}$ of $(U,V)$ with respect to $\mu\otimes \mu$ is defined by 
$$f_{U,V}(u,v)=\det(u)^p\exp(-\frac{1}{2}(\langle (\mathrm P(v)a)^{-1},u\rangle+\langle \nu_2,u^{-1}\rangle))\mbox{fct}(v),$$ 
$u,v\in E_+$, which gives the proposition.
\end{proof}
Let $\mathcal P$ be the Markov kernel of the Markov chain $( \Lambda_n,L_n)_{n \ge 0}$ defined in definition \ref{maindefi} and $\mathcal Q$ be the Markov kernel on $E_+\times E_+$ such that for $a\in E_+$, $\mathcal Q(a,.)$ is the law of $P( w + X^{-1})(a)$ where the $w$ is distributed according to $\mathrm{GIG}_{E_+}\left( p ; e, e \right)$ and $X$ is distributed according to $\mathrm{GIG}_{E_+}\left( p ; a^{-1}, e \right)$.
Now $\mathcal K$ is the Markov kernel from $E_+$ to $E_+\times E_+$ such that for all $a\in E_+$, 
$$
   \mathcal K(a, dx\, dy)=  \delta_a(dx) \ \mathrm{GIG}_{E_+}\left( p ; a^{-1}, e \right) (dy) \ .
$$ 
The following Theorem takes the historical point of view of Rogers--Pitman of intertwining \cite{rogers1981markov}. It implies in particular the Main Theorem \ref{thm:main_discrete}. The Markov kernel $\Phi$ from $E_+\times E_+$ to $E_+$ is defined by 
$$ \Phi(f) =f\circ \phi \ ,$$
for all bounded measurable real function $f$ on $E_+$, where $\phi(\lambda,x)=\lambda$.

\begin{thm}[Intertwining of Markov kernels and Markov property]
\label{thm:intertwining} 
We have \begin{enumerate}
\item $\mathcal K\circ \Phi=I$
\item 
$ \mathcal K \circ\mathcal P= \mathcal Q \circ\mathcal K$
\end{enumerate} 
\end{thm}
\begin{proof} Proposition \ref{first-step} with $\nu_1=\nu_2=e$ implies that for $a\in E_+$, $\mathcal K\circ \mathcal P(a,\cdot)$ and $\mathcal Q\circ\mathcal K(a.\cdot)$ are both the law of $(\mathrm P( w + X^{-1})(a),\mathrm P(w)(X)+w)$ where $w$ and $X$ are respectively distributed according to $\mathrm{GIG}_{E_+}\left( p ; e, e \right)$ and $\mathrm{GIG}_{E_+}\left( p ; a^{-1}, e \right)$.
\end{proof}

This last theorem proves Theorem \ref{thm:main_discrete}. Indeed, according to \cite[Theorem 2]{rogers1981markov} it follows from the previous theorem and the fact that $w_1$ is distributed according to $\mathcal K(e,(e,dx))$. While the intertwining property classically implies the Markov property, it is in fact equivalent to the conditional distribution property, as made explicit in \cite[Theorems 6.7.3 and 6.7.4]{chhaibi2013littelmann}.

\subsection{Proof of Theorem \ref{thm:dufresne_discrete}: The Dufresne identity} 
The integral process $\left( I_n \ ; \ n \ge 0 \right)$ is defined in definition \ref{defiint}. We now write the process according to the point of view of Chamayou and Letac \cite{chamayou1999additive}. 
We define the random fonction $F_n$ on $E_+$ by 
$$F_n(x)=\mathrm P(w^{-1}_n)(x)+w^{-1}_n, \, x\in E_+ \ .$$
For $n\ge 1$, we let $I_n(x)=F_1\circ\dots\circ F_n(x)$ for $x\in E_+$, so that 
\begin{align}
   I_n(x) 
& = \mathrm P(w^{-1}_1)\dots \mathrm P(w^{-1}_n)(x)+\sum_{k=0}^{n-1}P(w_1^{-1})\dots \mathrm P(w_k^{-1})(w_{k+1}^{-1}) \nonumber\\
& = \mathrm P(w^{-1}_1)\dots \mathrm P(w^{-1}_n)(x)+I_n \ .
\label{eq:InX_to_In}
\end{align}
The process $Y = \left( Y_n := F_n\circ\dots\circ F_1(x) \ ; \ n\ge 0 \right)$ is a Markov process starting from $x\in E_+$. The following proposition is an immediate consequence of \cite[Theorem 3.1]{letac2000indepedence} and the identity of Proposition \ref{magic-identity}. For the convenience of the reader we provide a complete proof. 

\begin{proposition}
\label{proposition:mes-inv} 
Suppose that $p>\dim E/r-1$. Then the inverse Wishart distribution $\mathrm{Inv}_{\#}(\gamma_{p,e})$ is an invariant probability measure of $Y$.
\end{proposition}
\begin{proof} 
Let $X$ be a random variable independent of $w_1$ and distributed according to the inverse Wishart distribution $\mathrm{Inv}_{\#}( \gamma_{p,e} )$.
We let $$U=\mathrm P(w_1^{-1})(X)+w_1^{-1} \textrm{ and } V=w_1^{-1}+X^{-1}\ .$$ As in Proposition \ref{first-step} one obtains the density $f_{U,V}$ of $(U,V)$ with respect to $\mu\otimes \mu$ as
$$ f_{U,V}(u,v)
   = \det(u)^{-p} \exp\left( -\frac{1}{2}\langle e,u^{-1}\rangle \right)
   F(v) \ ,
$$
for $u,v\in E_+$ and for a certain function $F$. This completes the proof of the proposition.
\end{proof}

By Theorem 2.4 of \cite{bougerol1992strict}, the existence of a stationary measure implies the convergence of the sequence $(I_n \ ; \ n\ge 0)$. We give a short proof of this fact which is available in the particular framework symmetric cones.
\begin{proposition} 
Suppose that $p>\dim E/r-1$. Then the series $$\sum_{k=0}^\infty \mathrm P(w^{-1}_1)\dots \mathrm P(w^{-1}_{k})(w^{-1}_{k+1})$$ is finite almost surely and distributed according to an inverse Wishart distribution $\mathrm{Inv}_{\#}( \gamma_{p,e} )$.
\end{proposition} 
\begin{proof} 
{\bf Step 1: Almost sure convergence.}

Notice that if $(a_n \ ; \ n \in \N )$ is a sequence of elements in $E_+$ then the series $\sum a_n$ converges absolutely in $E_+$ if and only if the series $\sum \Tr(a_n)$ converges absolutely in $\R_+$, which in turn holds if and only if $\sum \Tr(a_n)$ is bounded. Indeed the second equivalence holds because each term is positive. As for the first equivalence, the spectral decomposition implies that for $x \in E_+$, $\Tr x = \sum_i \lambda_i$ and $\|x\| = \sqrt{\sum_i \lambda_i^2}$ where the $\lambda_i$'s are positive eigenvalues. As such the first equivalence is basically a consequence of the equivalence between the $1$-norm and the $2$-norm in finite dimensions. In any case, we have
\begin{align}
    \label{eq:cv_iff}
    (I_n \ ; \ n \geq 1) \textrm{ converges}
    \textrm{ iff }
    (\Tr(I_n) \ ; \ n \geq 1 ) \textrm{ bounded } \ .
\end{align}

Now, let $X$ be a random variable independent from $(w_n)_{n\ge 1}$ distributed according to the inverse Wishart distribution $\mathrm{Inv}_{\#}( \gamma_{p,e} )$ on $E_+$. Thanks to Eq. \eqref{eq:InX_to_In}, we have
$$
    \Tr\left( I_n(X) \right) \geq \Tr \left( I_n \right) \ .
$$
By combining this fact with Eq. \eqref{eq:cv_iff}, we have
\begin{align*}
    \P\left( (I_n \ ; \ n \geq 1) \textrm{ converges} \right)
& = \P\left( (\Tr(I_n) \ ; \ n \geq 1 ) \textrm{ is bounded} \right)\\
& \geq \P\left( (\Tr(I_n(X)) \ ; \ n \geq 1 ) \textrm{ is bounded} \right)\\
& = \P(\cup_m\cap_n( \Tr(I_n(X) ) <m))\\
& = \sup_m \inf_n\P( \Tr(I_n(X)) <m) \\
& = \sup_m \P( \Tr(X) <m) = 1 \ .
\end{align*}
The penultimate equality comes from Proposition \ref{proposition:mes-inv} and from the equality in law of $I_n(X)$ and the stationary $Y_n$ (when $Y$ starts from $X$).

\medskip

{\bf Step 2: Identification of the limiting law.}
We have already proved that the series converges almost surely, which implies that $$ \mathrm P(w^{-1}_1)\dots \mathrm P(w^{-1}_{n})(w^{-1}_{n+1})$$ goes to $0$ almost surely when $n$ goes to infinity. Let $w_0$ be distributed as $w_{n+1}$ and independent from $(w_n)_{n\ge 1} $. Then $\left( \mathrm P(w^{-1}_1)\dots \mathrm P(w^{-1}_{ n})(w^{-1}_{0}) \ ; \ n \in \N \right)$ converges in probability to $0$. One chooses a deterministic subsequence such that $\mathrm P(w^{-1}_1)\dots \mathrm P(w^{-1}_{\varphi(n)})(w^{-1}_{0})$ which goes to $0$ almost surely. As $w_0$ is independent from $(w_n)_{n\ge 1} $, for almost every $x\in E_+$, $\mathrm P(w^{-1}_1)\dots \mathrm P(w^{-1}_{\varphi(n)})(x)$ goes to $0$ almost surely when $n$ goes to infinity. 

Now, thanks to Eq. \eqref{eq:InX_to_In}, the limits of $I_{\varphi(n)}$ and $I_{\varphi(n)}(X)$ coincide almost surely. Since $I_{\varphi(n)}(X) \eqlaw X$ using the same argument as the end of Step 1, we are done with the proof.
\end{proof}

\section{Proofs in the continuous time setting }
\label{section:continuous_time}

As announced, we invoke invariance principles from \cite{stroock1973limit} to deduce Theorems \ref{thm:main_continuous} and \ref{thm:dufresne_continuous} from Theorems \ref{thm:main_discrete} and \ref{thm:dufresne_discrete}. Indeed, the group $G$ is a connected Lie subgroup of $GL(E)$ so that the results of \cite{stroock1973limit} are available here. 

We organize this section as follows.
\begin{itemize}
    \item In Subsection \ref{subsection:algebraic_structureGK}, we start by prerequisites on the algebraic structure of Jordan algebras.
    \item In Subsection \ref{subsection:discrete_approximations}, we introduce discrete time processes, following the framework developed in the previous Section \ref{section:discrete_time}.
    \item Only then we prove in Subsection \ref{subsection:scaling_limit} the scaling limit to continuous time Markov processes, among which a left-invariant diffusion on $G$, driving the other processes.
    \item We deduce in Subsection \ref{subsection:proofs_theo_cont_1} Theorems \ref{thm:main_continuous} and \ref{thm:dufresne_continuous} for an intermediate Brownian motion on $G$ driven by a Brownian motion on $\mathfrak p$ from their discrete time versions.
    \item Finally we deduce in Subsection \ref{subsection:proofs_theo_cont_2} Theorems \ref{thm:main_continuous} and \ref{thm:dufresne_continuous}.
\end{itemize}


\subsection{Algebraic structure related to a Jordan algebra} 
\label{subsection:algebraic_structureGK}

Let $\gfrak$ be the Lie algebra of $G$. 
Notice that $G \subset \mathrm{Aut}(E)$ and $\gfrak \subset \mathrm{End}(E)$. 
If we write $K = G\cap O(E)$, where $O(E)$ is the orthogonal group with respect to the inner product $\langle \cdot,\cdot\rangle$ on $E$ then it is known that 
$$ K = \left\{ g\in G \ ; \ g(e) = e \right\} \ . $$
Moreover, $G / K$ is a homogeneous symmetric space and the mapping $gK \mapsto g(e)$ identifies the symmetric space $G / K$ with the cone $E_+$. 

The Lie algebra $\gfrak$ has a Cartan decomposition which is described in \cite[Chapter I, Paragraph 8]{9ddeeb4c-7522-3145-b466-0c73883a786a}. 
One has $\gfrak=\kfrak\oplus \mathfrak p$ where $\kfrak$ is the Lie algebra of $K$ given by 
$$ \kfrak := \left\{ X\in \mathfrak{g} \ ; \ X(e) = 0 \right\} \ ,$$
and 
$$ \mathfrak p := \left\{ \mathrm L(x) \ ; \ x \in E \right\} \ .$$ 
We fix a complete system $c_1, \dots , c_r$ of primitive orthogonal idempotents in $E$. The maximal abelian subalgebra in $\mathfrak p$ is the subspace $\mathfrak a$ of $\mathfrak p$ defined by 
$$\mathfrak a = \{\mathrm L(a) \ ; \ a \in R\}$$ 
where $R=\{a=\sum_{i=1}^r a_ic_i \ ; \ a_1,\dots,a_r\in \mathbb \R\}$. The roots of $\gfrak$ are the forms $\alpha_{ij}$, $i,j\in \{1,\cdots,r\}$, $i\ne j$, defined on $\mathfrak a$ by 
$$\alpha_{ij}(\mathrm L(a))=\frac{1}{2}(a_j-a_i),$$
where $a=\sum_{i=1}^r a_i c_i$. The root space decomposition of $\gfrak$ is \cite[Prop. VI.3.3]{faraut1994analysis}
$$ \gfrak = \mathfrak m \oplus \mathfrak a \oplus_{i\ne j}  \gfrak_{ij} \ , $$
where $\gfrak_{ij}$ is the root space associated to $\alpha_{ij}$ and 
$$\mathfrak m=\{X\in \kfrak \ ; \ X(a)=0, \ a\in R\} \ .$$ 
The dimension of $\gfrak_{ij}$ does not depend on $i,j$. It is denoted by $d$ and called the degree of the Jordan algebra $E$. One has $\dim E=r+dr(r-1)/2$. We equip $\mathfrak p$ with an inner product $\langle \cdot,\cdot\rangle$ defined by $\langle \mathrm L(x),\mathrm L(y)\rangle=\langle x,y\rangle$, $x,y\in E$. It coincides up to a scaling factor with the Killing form on $\mathfrak p'=\{\mathrm L(x)\in \mathfrak p \ ; \ \Tr(x)=0\}$. Thus the vector $H_\rho$ which represents the Weyl vector in $\mathfrak a$ is defined by 
$$ H_\rho = \frac{d}{4}\sum_{i=1}^r(2i-r-1)L(c_i) \ .$$
 
\subsection{Discrete time approximations of the processes}
\label{subsection:discrete_approximations}
Let $\lambda_0\in E_+ $ and $\ell_0\in E_+ $.
\begin{definition}
\label{maindefi-n}
Let $n$ be a positive integer and $\left(w_k^{n} \ ; \ k \ge 1\right)$ be a sequence of i.i.d. random variables on $E_+$ distributed according to $\mathrm{GIG}_{E_+}\left( p ; ne, ne \right)$. The Markov chain $\left( L^{n}_k \ ; \ k \ge 0\right)$ and the random processes 
$\left( \Lambda^{n}_k \ ; \ k \ge 0\right)$,
$\left( I_k^n \ ; \ k\ge 0\right)$ defined by
\begin{align*}
L^{n}_{k+1}       := & \ \mathrm P(w^{n}_{k+1})(L^{n}_k)+\frac{1}{n}w^{n}_{k+1} \ , \\
\Lambda^{n}_{k+1} := & \ \mathrm P\left( w^{n}_{k+1}+\frac{1}{n}(L^{n}_k)^{-1} \right)(\Lambda^{n}_{k}) \ , \\
I_k^{n}           := & \ 
   \frac{1}{n}\sum_{i=0}^{k-1}\mathrm P((w^{n}_1)^{-1})\dots \mathrm P((w_i^n)^{-1})(w_{i+1}^n)^{-1} \ ,
\end{align*}
with starting points $L^{n}_0=\ell_0$, $\Lambda^{n}_0=\lambda_0$. 
\end{definition}
The following proposition gives other definitions of $\Lambda^n$ and $L^n$ in function of $I^n$ which are more efficient to obtain a continuous time version of Theorem \ref{thm:main_discrete}.

\begin{proposition} 
\label{proposition:chain_expressions}
For $k\ge 0$,
\begin{enumerate} 
\item $L^{n}_k=\mathrm P(w^{n}_k)\dots \mathrm P(w_1^{n})(\ell_0+I^{n}_k),$
\item $\Lambda^n_k=\mathrm P(L^n_{k})\mathrm P((w^n_{k})^{-1})\dots \mathrm P((w^n_{1})^{-1})\mathrm P(\ell_0^{-1})(\lambda_0)$.
\end{enumerate}
\end{proposition} 
\begin{proof} 
First notice that identities are true for $k=0$. Points (1) and (2) follow from the fact that the sequence 
$$(\mathrm P(w^n_k)\dots \mathrm P(w^n_1)(\ell_0+I^n_k)\ ;\ k\ge 1) $$ 
and 
$$(\mathrm P(L^n_{k})\mathrm P((w^n_{k})^{-1})\dots \mathrm P((w^n_{1})^{-1})\mathrm P(\ell_0^{-1})(\lambda_0)\ ;\ k\ge 0)$$ satisfy the same recurrence relations as $L^n$ and $\Lambda^n$. It is immediate for the first sequence. For the second, we let for $k\ge 0$, $$Y_k=\mathrm P(w_k^n)\dots \mathrm P(w_1^n)\mathrm P(\ell_0)(\lambda_0^{-1}).$$ From Proposition \ref{proposition:forlim} we obtain
\begin{align*}
\mathrm P(L^n_{k+1})(Y_{k+1}^{-1})&=\mathrm P(\mathrm P(w^n_{k+1})L^n_k+\frac{1}{n}w^n_{k+1})\mathrm P((w_{k+1}^n)^{-1})(Y_k^{-1})\\
&=\mathrm P(w^n_{k+1})\mathrm P(L_k+\frac{1}{n}(w^n_{k+1})^{-1}) (Y_k^{-1})\\ 
&=\mathrm P(w^n_{k+1}+\frac{1}{n}(L^n_k)^{-1})\mathrm P(L^n_k)(Y^{-1}_k),
\end{align*}
the last equality coming from Proposition \ref{proposition:forlim}.
Thus the sequence $$\left(\mathrm P(L^n_k)(Y^{-1}_k)\ ; \ k\ge 1\right)$$ satisfies the same recurrence relation as $\Lambda^n$, which finishes the proof.
\end{proof}
The following proposition follows from Proposition \ref{first-step}.

\begin{proposition}
\label{first-step-small-step}
Let $a\in E_+$. Let $X$ and $w$ be two independent random variables such that the law of $X$ and $w$ are respectively $\mathrm{GIG}_{E_+}\left( p ; a^{-1}, e \right)$ and $\mathrm{GIG}_{E_+}\left( p ; ne, ne \right)$. 
Then the law of $\mathrm P(w)(X)+\frac{1}{n}w$ given $\mathrm P(w+\frac{1}{n}X^{-1})(a)=b$ is $\mathrm{GIG}_{E_+}\left( p ; b^{-1}, e \right)$.
\end{proposition}

Proposition \ref{first-step-small-step} implies as previously the following, where $\mathcal P^{(n)}$ is the Markov transition kernel of $\left((\Lambda_k^n,L^{n}_k) \ ; \ k\ge 0\right)$ and $\mathcal K$ is defined as in Theorem \ref{thm:intertwining}.
 \begin{proposition}\label{intertwin-small-step} For $\lambda_0=\ell_0=0$, the random process $\left(\Lambda_k^n; k\ge 1\right)$ is a Markov chain on $E_+$ whose Markov kernel $\mathcal Q^{(n)}$ is defined by 
 $$\mathcal Q^{(n)}=\mathcal K\circ \mathcal P^{(n)}\circ \Phi$$ and satisfies the intertwining 
 $\mathcal Q^{(n)}\circ \mathcal K=\mathcal K\circ \mathcal P^{(n)}.$
 \end{proposition} 

\begin{proposition}\label{proposition:mes-inv-small-step} 
Suppose that $p<1-\dim E/r$. Then the inverse Wishart distribution $\mathrm{Inv}_{\#}(\gamma_{-p,e})$ is an invariant probability measure of the Markov chain $L^n$.
\end{proposition}

\subsection{Scaling limit to continuous-time processes}
\label{subsection:scaling_limit}

For a differentiable function $f$ on $G$ and $X\in \gfrak$ we set
$$X f(g)=\frac{d}{dt}_{\vert{t=0}}f(ge^{tX}), \, g\in G.$$ We consider an orthonormal basis $(X_1,\dots,X_{\dim(E)})$ of $\mathfrak p$. Let $(\widetilde{g}_t\ ;\ t\ge 0)$ be a left invariant diffusion on $G$ with infinitesimal generator 
$$2\sum_{i=1}^{\dim E}X_i^2+2pL(e) \ .$$
Using the Stratonovich integration convention, this means that we are given standard real Brownian motions $(B^{(1)},\dots,B^{(\dim(E))})$ such that
\begin{align}
    \label{eq:def_tildeg}
    d\widetilde{g}_t = & \ \widetilde{g}_t \circ \left( 2 \sum_{i=1}^{\dim E} X_i dB_t^{(i)}  + 2 p \Lrm(e)dt \right) \ .
\end{align}
The diffusion $\widetilde{g}$ is a hypo--elliptic Brownian motion on $G$. Notice that $\widetilde{g}_t K \in G/K$ is nothing but the left Brownian motion on $G/K$ with covariance $4 \id_{\mathfrak{p}^{\otimes 2}}$ and drift $ 2p L(e)$. For more details, we refer to the paper of Marie-Paule and Paul Malliavin \cite{malliavin74}. 
Furthermore define the random processes $\widetilde \iota=(\widetilde\iota_t \ ; \ t\ge 0)$, $\widetilde\ell=(\widetilde\ell_t \ ; \ t\ge 0)$ and $\widetilde\lambda=(\widetilde\lambda_t \ ; \ t\ge 0)$ by 
\begin{align}\label{definition:tilde_processes}
\widetilde\iota_t=\int_0^t(\widetilde{g}_s^*)^{-1}(e)\, ds,\,\,\, \widetilde\ell_t=\widetilde{g}_t^*(\ell_0+\widetilde\iota_t),\, \,\, \widetilde\lambda_t=\widetilde{g}_t^*P(\ell_0+\widetilde\iota_t)P(\ell_0^{-1})(\lambda_0),
\end{align} for $t\ge 0.$
It is important to notice that these processes are not the ones defined in Eq. \eqref{threeproc}. The latter processes, which appear in the statement of the main theorem, are associated to a Brownian motion on $G$ whereas these ones with a tilde are associated to a hypo--elliptic Brownian motion on $G$.
Furthermore, notice that we obtain processes starting from $0$ as in \eqref{threeproc} if we let $\lambda_0=\ell_0^2 \rightarrow 0 \in E \backslash E_+$ in the previous definitions. 

The main result of this Subsection gives these processes as diffusive limits from the discrete case, using an invariance principle in the context of Lie groups. For the convenience of the reader, this invariance principle due to Daniel Stroock and Srinivasa Varadhan \cite{stroock1973limit} is explained and specialized to our context in Appendix \ref{section:stroock_varadhan}. Our statement for this Subsection is as follows.

\begin{proposition} Let $T\ge 0$. If $\widetilde{G}_k^n=P(w_1^n)\dots P(w_k^n)$, for $k\ge 0$, then 
\label{proposition:conv} 
the sequence of processes 
$$ \widetilde{g}^{(n)} = \left(\widetilde{g}^{(n)}_t := \widetilde{G}_{\lfloor nt\rfloor}^n \ ; \ t\in [0,T] \right), \, n\ge 1, $$ 
converges in law in $\mathcal D([0,T],G)$ equipped with the Skorokhod topology towards the diffusion $(\widetilde{g}_t \ ; \ t \in [0,T])$. Moreover, the sequence of processes $$\left(\Lambda_{\lfloor nt\rfloor}^n,L_{\lfloor nt\rfloor}^n,I_{\lfloor nt\rfloor}^n \ ; \ t\in [0,T]\right), \, n\ge 1,$$ converges in law in $\mathcal D([0,T],E^{\times 3})$ towards $\left(\widetilde\lambda_t,\widetilde\ell_t,\widetilde\iota_t \ ; \ t\in [0,T]\right)$.

\end{proposition}
\begin{proof} The weak convergence of $\widetilde{g}^{(n)}$ to $\widetilde{g}$ will follow from Theorem \ref{thm:adaptation_of_stroock_varadhan}. Carefully checking its hypotheses is left as Step 2. We start with the second point that is as our Step 1.

\medskip

{\bf Step 1: Scaling limit of the triplet $(\Lambda, L, I)$.} 
We assume that $\widetilde{g}^{(n)}$ converges in law to $\widetilde{g}$. First notice that Slutsky's Theorem implies that 
$$ \left(\widetilde g_t^{(n)}, (\widetilde g_t^{(n)*})^{-1}((w_{\lfloor nt\rfloor+1}^n)^{-1})\, ;\, t\in [0,T]\right)$$ converges in law towards $$ \left(\widetilde g_t, (\widetilde g_t^*)^{-1}(e)\, ;\, t\in [0,T]\right).$$ The function $F$ defined on $\mathcal D([0,T],G)\times \mathcal D([0,T],E)$ by 
$$F(f,x)(t)=\left(f(t),\int_0^t x(s)\, ds\right),$$ $f\in\mathcal D([0,T],G), x\in\mathcal D([0,T],  E), \, t\in \R_+$,
 is continuous. Thus the classical mapping theorem as in \cite[Theorem 2.7]{billingsley2013convergence} implies that 
 $$\left(\widetilde g^{(n)}_t,\int_0^t (\widetilde g_s^{(n)*})^{-1}((w_{\lfloor ns\rfloor+1}^n)^{-1}))\, ds\ ;\ t\in [0,T]\right)$$ converges in distribution towards $(\widetilde g,\widetilde \iota)$. Besides one has for $t\in [0,T]$
\begin{align*}
    I_{\lfloor nt\rfloor}^n&=\int_0^t (\widetilde g_s^{(n)*})^{-1}((w_{\lfloor ns\rfloor+1}^n)^{-1})\, ds-\frac{nt-\lfloor nt\rfloor}{n}(\widetilde g^{(n)*}_{t})^{-1}((w_{\lfloor nt\rfloor+1}^n)^{-1})\ .
\end{align*} 
Using Slutsky's Theorem again one obtains that 
 $$\left(\widetilde g^{(n)}_t,I_{\lfloor nt\rfloor}^n\ ;\ t\in [0,T]\right)$$ converges in distribution towards $(\widetilde g,\widetilde \iota)$.
 As from Proposition \ref{proposition:chain_expressions}, for $t\in [0,T]$,
$$L_{\lfloor nt\rfloor}^n=\widetilde g^{(n)*}_t(\ell_0+I_{\lfloor nt\rfloor}^n)\textrm{ and } 
\Lambda_{\lfloor nt\rfloor}^n=\widetilde g^{(n)*}_tP(\ell_0+I_{\lfloor nt\rfloor}^n)P(\ell_0^{-1})(\lambda_0)\ ,$$
step 1 is done. 

\medskip

{\bf Step 2: Asymptotics of the distribution $\mathrm{GIG}_{E_+}\left( p ; ne, ne \right)$.} 
We now check the hypotheses of the Stroock--Varadhan criterion which is stated in Appendix as Theorem \ref{thm:adaptation_of_stroock_varadhan}.

\medskip

{\it Step 2.1: $\sqrt{n} \log \left( w^{(n)} \right)$ converges in law towards a standard Gaussian on $E$, with all moments remaining bounded. } 

From the definition of the GIG distribution in Eq. \eqref{eq:def_GIG}, we have for any positive function $f$
   \begin{align*}
               & \E\left( f( \sqrt{n} \log w^{(n)} ) \right) \\
       = & \frac{1}{2K(p;ne,ne)} \int_{E_+} f(\sqrt{n} \log w)\det(w)^{p-\dim E/r}\exp\left( -\frac{n}{2} \langle e,w+w^{-1} \rangle \right) \ dw \\
       = & \frac{1}{2K(p;ne,ne)} \int_E f( \sqrt{n} x )
       \exp\left( (p-\dim E/r)\tr(x) \right) \\
         & \quad \quad \quad \quad \quad \quad \quad \quad 
            \exp\left( -\frac{n}{2} \langle e, e^{x}+e^{-x} \rangle \right) 
       \vert\det J(x)\vert \ dx \ ,
   \end{align*}
   where on the last step, we performed the change of variables $w = \exp(x)$. In particular $J(x)$ is the Jacobian of $x\mapsto \exp(x)$. Because of the Laplace method, the probability mass concentrates around $0$ with Gaussian fluctuations at scale $\sqrt{n}$. 

   Doing the computation in more details, let us write
   \begin{align}
    \label{eq:approximate_density}
      \E\left( f( \sqrt{n} \log w^{(n)} ) \right)
    = & \frac{ (2\pi n)^{\half \dim E} }{2K(p;ne,ne) e^{n r} } \int_E f( \sqrt{n} x )
       \phi_n(x) \ dx \ ,
   \end{align}
   with
   $$
        \phi_n(x) = \frac{\vert \det J(x) \vert}
                         {(2\pi n)^{\half \dim E}}
        \exp\left( (p-\dim E/r)\tr(x) -\frac{n}{2} \langle e, e^{x}+e^{-x} - 2 e\rangle \right) \ .
   $$

   Clearly, as $x \rightarrow 0$, we have $\det J(x) \rightarrow 1$ and $\tr(x) \rightarrow 0$, so that
   \begin{align}
   \label{eq:quantitative1}
        \phi_n(x) & = \frac{1 + o(1)}{(2\pi n)^{\half \dim E}}
                    \exp\left( -\frac{n}{2} \langle e,x^2 \rangle (1+o(1)) \right) \ .
   \end{align}
   Also, because of the scalar inequality $e^{t} + e^{-t} - 2 \geq t^2$, we also have for $n \geq 2$
   \begin{align}
   \label{eq:quantitative2}
        \phi_n(x) & \leq \frac{1}{(2\pi n)^{\half \dim E}}
                    \exp\left( -\frac{n-1}{2} \langle e,x^2 \rangle \right) 
                    \phi_1(x)
                    \ .
   \end{align}
   This leads us to splitting the integral and invoking the quantitative estimates in Eq. \eqref{eq:quantitative1} and the inequality \eqref{eq:quantitative2}
   \begin{align*}
      & \E\left[ f( \sqrt{n} \log w^{(n)} ) \right] \\
    = & \frac{ (2\pi n)^{\half \dim E} }{2K(p;ne,ne) e^{nr} } \left(
        \int_{\|x\| \leq n^{-\frac14} } f( \sqrt{n} x ) \phi_n(x) \ dx 
        + 
        \int_{\|x\| > n^{-\frac14} } f( \sqrt{n} x ) \phi_n(x) \ dx 
        \right) \\
    = & \frac{ (2\pi n)^{\half \dim E} }{2K(p;ne,ne) e^{nr} } \Big(
        \int_{\|x\| \leq n^{-\frac14} } f( \sqrt{n} x ) \frac{1 + o(1)}{(2\pi n)^{\half \dim E}}
                    \exp\left( -\frac{n}{2} \langle e,x^2 \rangle (1+o(1)) \right) \ dx \\
      & \quad
        + 
        \Oc\left( \int_{\|x\| > n^{-\frac14} } f( \sqrt{n} x ) \frac{1}{(2\pi n)^{\half \dim E}}
                    \exp\left( -\frac{n-1}{2} \langle e,x^2 \rangle \right) 
                    \phi_1(x) \ dx 
        \right)
        \Big) \\
    = & \frac{ (2\pi n)^{\half \dim E} }{2K(p;ne,ne) e^{nr} } \Big(
        \int_{\|y\| \leq n^{\frac14} } f( y ) \frac{1 + o(1)}{(2\pi )^{\half \dim E}}
                    \exp\left( -\frac{1}{2} \| y \|^2 (1+o(1)) \right) \ dy \\
      & \quad
        + 
        \Oc\left( n^{-\half \dim E}
                  \sup_{\|x\| > n^{-\frac14}} 
                  \left| f(\sqrt{n}x) \exp\left( -\frac{n-1}{2} \| x \|^2 \right) \right| 
        \int \phi_1
        \right)
        \Big) \ .
   \end{align*}
   Because $\phi_1$ is integrable, $\int \phi_1$ can be absorbed in the $\Oc$'s implicit constant. As such, we obtain:
   \begin{align}
   \label{eq:formula_for_any_f}
      & \E\left[ f( \sqrt{n} \log w^{(n)} ) \right] \nonumber \\
    = & \frac{ (2\pi n)^{\half \dim E} }{2K(p;ne,ne) e^{nr} } \Big(
        \int_{\|y\| \leq n^{\frac14} } f( y ) \frac{1 + o(1)}{(2\pi )^{\half \dim E}}
                    \exp\left( -\frac{1}{2} \| y \|^2 (1+o(1)) \right) \ dy \nonumber\\
      & \quad
        + 
        \Oc\left( n^{-\half \dim E}
                  \sup_{\|x\| > n^{-\frac14}} 
                  \left| f(\sqrt{n}x) \exp\left( -\frac{n-1}{2} \| x \|^2 \right) \right| 
        \right)
        \Big) \ .
   \end{align}
   By picking $f=1$ in Eq. \eqref{eq:formula_for_any_f}, we see that the first in the integrand yields asymptotically $1$, while the second is negligible because $\phi_1$ is integrable, by definition of the GIG. As such, we obtain the asymptotic for the normalization constant
   \begin{align}
       \label{eq:normalization_constant}
       \frac{ (2\pi n)^{\half \dim E} }{2K(p;ne,ne) e^{nr} }
       \sim_{n \rightarrow \infty} 1 \ .
   \end{align}
   Morevoer, by picking a general positive $f$ in Eq. \eqref{eq:formula_for_any_f}, we have
      \begin{align*}
      & \E\left[ f( \sqrt{n} \log w^{(n)} ) \right] \\
    = & \left( 1 + o(1) \right) \left(
        \int_{\|y\| \leq n^{\frac14} } f( y ) \frac{1 + o(1)}{(2\pi )^{\half \dim E}}
                    \exp\left( -\frac{1}{2} \| y \|^2 (1+o(1)) \right) \ dy \right)\\
      & \quad
        + 
        \Oc\left( n^{-\half \dim E}
                  e^{ -\half \sqrt{n}} \|f\|_\infty \right) \ .
   \end{align*}
   This is the announced Gaussian convergence.
   
   Finally, by picking $f: x \in E \mapsto \|x\|^k \in \R_+$ in Eq. \eqref{eq:formula_for_any_f}, we see that all moments are bounded.
   
   \medskip

{\it Step 2.2: Proof of $\lim_{n \rightarrow \infty} \E\left( n \log w^{(n)} \right) = pe$. } 
Going back to the previous expression in Eq. \eqref{eq:approximate_density}, but this time invoking the asymptotic behavior of the normalization constant given in Eq. \eqref{eq:normalization_constant}, we have
    \begin{align*}
                \E\left[ n \log w^{(n)} \right] 
        = & (1+o(1)) \int_E n x \phi_n(x) \ dx \ .
   \end{align*}
   Once again, we can perform the same splitting, as the integral on $\{ \| x \| > n^{-\frac14} \}$ will decay exponentially. Indeed, the additionnal term $n$ plays little role. In the end, we have
    \begin{align*}
          & \E\left[ n \log w^{(n)} \right] \\
        = & o(1) + (1+o(1)) \int_{\| x \| < n^{-\frac14}} n x \phi_n(x) \ dx \ .
   \end{align*}
   
   We conclude by finer Taylor expansions for $\phi_n$. As $x \rightarrow 0$, we have $J(x) = \exp\left( L(x) + \Oc(L(x)^2) \right)$ and $\tr(L(x))=\frac{\dim E}{r} \tr(x)$ thanks to \cite[Proposition III.4.2 (i)]{faraut1994analysis}, one gets
   $$ \vert \det J(x) \vert
   = \exp\left( \tr(L(x)) + \Oc(x^2) \right)
   = (1+\Oc(x^2))\exp\left( \frac{\dim E}{r} \tr(x) \right)
    $$
   and thus
   \begin{align*}
                & \E\left[ n \log w^{(n)} \right] \\
        = & o(1) + (1+o(1)) \int_{\| x \| < n^{-\frac14}} 
        \frac{n x 
        \exp\left( p \tr(x) \right)}
        {(2\pi n)^{\half \dim E}}
        \exp\left( -\frac{n}{2}\langle e, e^{x}+e^{-x} \rangle \right) \ dx \ .
    \end{align*}
    Here only even terms in the integral remain, so that we can write
    \begin{align*}
        & \E\left[ n \log w^{(n)} \right] \\
        = & o(1) + (1+o(1)) \int_{\| x \| < n^{-\frac14}} 
        \frac{n x 
        \left( \exp\left( p \tr(x) \right) - 1 \right)}
        {(2\pi n)^{\half \dim E}}
        \exp\left( -\frac{n}{2}\langle e, e^{x}+e^{-x} \rangle \right) \ dx\\
        = & o(1) + (1+o(1)) \int_{\| x \| < n^{-\frac14}}
        \frac{ n x  p \tr(x) (1 + o(1) ) }
             { {(2\pi n)^{\half \dim E}} }
        \exp\left( -\frac{n}{2} \langle e,x^2 \rangle (1+o(1)) \right) \ dx \ .
   \end{align*}
    Upon performing the change of variable $y = \sqrt{n} x$, we have
   \begin{align*}
    & \E\left( n\log(w^n) \right) \\
  = & o(1) + (1+o(1)) \int_{\| y \| < n^{\frac14}}
        \frac{ y  p \tr(y) (1 + o(1) ) }
             { {(2\pi)^{\half \dim E}} }
        \exp\left( -\frac{1}{2} \| y \|^2 (1+o(1)) \right) \ dy \ .
   \end{align*}
   We conclude by Lebesgue's Dominated Convergence Theorem that
   \begin{align*}
    & \lim_{n \rightarrow \infty} \E\left( n\log(w^n) \right) \\
  = & \int_{E}
        y  p \tr(y)
        \exp\left( -\frac{1}{2} \| y \|^2 \right) \ \frac{dy}{{(2\pi n)^{\half \dim E}}} \\
  = & \ pe \ .
   \end{align*}
   Indeed $y = \tr(y) e + \left( y - \tr(y) e\right)$ is an orthogonal decomposition in $E$.
   \end{proof}

\subsection{Proofs of Theorems \ref{thm:main_continuous} and \ref{thm:dufresne_continuous} for intermediate processes }
\label{subsection:proofs_theo_cont_1}
We  first prove Theorems \ref{thm:main_continuous} and \ref{thm:dufresne_continuous} for the hypoelliptic Brownian motion $\widetilde{g}$ instead of $g$. First notice that  identities 
  \begin{align*}
      \widetilde\ell_{t} & = (\widetilde g_s^{-1}\widetilde g_{t})^* \left( \int_0^{t-s}(\widetilde g_s^{-1}\widetilde g_{s+u})^{-*}(e)\ du +\widetilde\ell_s \right) \ ,\\
      \widetilde \lambda_t&=P(\widetilde\ell_t)(\widetilde g^{-1}_s\widetilde g_t)^{-1}P(\widetilde\ell_s^{-1})(\widetilde\lambda_s)\ ,
 \end{align*}
 for 
 $0\le s\le t$, implies that the couple $(\lambda,\ell)$ is Markovian. We denote $\left(\mathcal P_t \ ; \ t\ge 0\right)$ its Markovian semi-group.
The following proposition implies in particular that Theorem \ref{thm:main_continuous} is true for $\widetilde{g}$, 
according to \cite[Theorem 2]{rogers1981markov}. The probability measure $\mathcal K(\lambda_0,\cdot)$ is defined as in Theorem \ref{thm:intertwining} for $\lambda_0\in E_+$.

\begin{proposition}
\label{proposition:intertwining_cont}
For all $t\ge 0$, $$\mathcal Q_t:=\mathcal K\circ \mathcal P_t\circ \Phi$$ is a semi-group which satisfies
$$ \mathcal K \circ \mathcal P_t = \mathcal Q_t \circ \mathcal K \ .$$
\end{proposition}

Let us now prove Theorem \ref{thm:dufresne_continuous} for $\widetilde{g}$. We shall use the following proposition.
\begin{proposition}
\label{convlem} 
If $p > -1 + \frac{\dim E}{r}$, then 
$\lim_{t\to \infty} (\widetilde{g}_t^*)^{-1}=0$ and
$ \widetilde i_\infty$ exists in $E_+$.
\end{proposition}
\begin{proof}
Let $(u_t \ ; \ t\ge 0)$ be a standard Brownian motion on $G$. Then it is well-known (see for example \cite[Theorem 3.4]{Taylor1991brownian}) that in $K\backslash G\slash K$, isotropic Brownian motion has a drift following the Weyl vector, that is to say
\begin{align}
    \label{eq:taylorboubou}
    u_t \underset{t\to\infty} = \exp\left( tH_\rho+o(t) \right) \ .
\end{align}
Both authors learned this fact from Philippe Bougerol, a while ago. We thank him for pointing this special drift during a private communication.

The process $((\widetilde{g}_t^*)^{-1}\ ; \ t\ge0)$ is a Brownian motion on $G/K$ with covariance $4 \id_{\mathfrak{p}^{\otimes 2}} $ and drift $-2p \id_{\mathfrak{gl}(E)}$. Then, Eq. \eqref{eq:taylorboubou} implies that in $K\backslash G\slash K$, when $t$ goes to infinity 
\begin{align*}
  \widetilde{g}_t  =& \exp(-2t(pL(e)-2H_\rho)+o(t))\\
   =& \exp(-2t\sum_{i=1}^r(p-\frac{d}{2}(2i-r-1))L(c_i)+o(t)) \ .
\end{align*}
The smallest exponent of decay is thus obtained from $i=r$ and the proposition follows as soon as $p > \frac{d}{2}(r-1)$. Recalling that the degree $d$ of the Jordan algebra satisfies $\dim E = r + \frac{d}{2}r(r-1)$, we have
$$
\frac{d}{2}(r-1) = \frac{1}{r}( \dim E - r ) = -1 + \frac{\dim E}{r} \ ,
$$
hence the result.
\end{proof}

\begin{proposition}
    \label{proposition:inv-mes-cont}
    Suppose that $p<1-\dim E/r$. Then the inverse Wishart distribution $\mathrm{Inv}_{\#}(\gamma_{-p,e})$ is an invariant probability measure of the Markov process $\widetilde\ell$.
\end{proposition}
\begin{proof} It follows from Proposition \ref{proposition:mes-inv-small-step} and Proposition \ref{proposition:conv}. 
\end{proof}

We can now prove Theorem \ref{thm:dufresne_continuous} for $(\widetilde{g}_t\ ;\ t\ge 0)$. It follows from the following proposition. 
\begin{proposition}
For any $t\ge 0$, $$\tilde\ell_t \overset{\Lc}{=}\int_0^t\widetilde{g}_s(e)\ ds+\widetilde g_t(\ell_0)\ .$$ Furthermore when $p<1-\dim E/r$, the perpetuity $\int_0^\infty\widetilde{g}_s(e)\ ds$ is finite almost surely and distributed according to an inverse Wishart distribution $\mathrm{Inv}_{\#}(\gamma_{-p,e})$. 
\end{proposition}
\begin{proof}
For the first part of the proposition we write 
\begin{align*}
    \widetilde\ell_t &=\widetilde g_t^*(\int_0^t (\widetilde g_s^*)^{-1}(e)\, ds+\ell_0)\\
    &=   \int_0^t \widetilde g_t^*(\widetilde g_{t-s}^*)^{-1}(e)\, ds+\widetilde g_t^*(\ell_0)\ ,
\end{align*}
 which proves the expected identity in law as $(\widetilde g_t^*(\widetilde g_{t-s}^*)^{-1}\, ;\, s\in[0,t])$ as the same law as $(\widetilde g_s^* \, ;\, s\in[0,t])$. For the second point, Eq. \eqref{eq:taylorboubou} implies that when $p<1-\dim E/r$ the perpetuity $\int_0^\infty \widetilde g_s(e)\ ds$ is finite almost surely and $\lim_{t\to \infty} \widetilde g_t=0$. The law of the perpetuity is deduced from Proposition \ref{proposition:inv-mes-cont}.
\end{proof}

\subsection{Proofs of Theorems \ref{thm:main_continuous} and \ref{thm:dufresne_continuous}}
\label{subsection:proofs_theo_cont_2}
 Let us now prove Theorems \ref{thm:main_continuous} and \ref{thm:dufresne_continuous} for the Brownian motion on $G$ considered in the theorems. First notice that the couple $(\lambda,\ell)$ is Markovian as $(\widetilde \lambda,\widetilde\ell)$.

\begin{lemma}
\label{lemma:cond_prod}
Let $\left( k_t \ ; \ t \ge 0 \right)$ be a left Brownian motion on $K$. Then for $t>0$, the conditional law of $k_t^*\left( \widetilde\ell_t \right)$ given $\left( k^*_t\left( \widetilde\lambda_s \right) \ ; \ s\le t \right)$ is given by the distribution $\mathrm{GIG}_{E_+}\left( p ; a^{-1}, e \right)$, where $a=k_t\widetilde\lambda_t$.
\end{lemma}
\begin{proof}
The second point of Theorem \ref{thm:main_continuous} proved for $\widetilde g$ and the fact that $k(e)=e$ and $(k(x))^{-1}=k(x^{-1})$ imply that for any $k\in K$, $t\in \R_+^*$, the conditional law of $k(\widetilde \ell_t)$ given $\widetilde\lambda_t$ is the distribution $\mathrm{GIG}_{E_+}\left( p ; (k\widetilde \lambda_t)^{-1}, e \right)$, from which the lemma follows.
 \end{proof}
 In a similar fashion to the factorization Lemma \ref{lemma:factorization_gk}, one has the following variant.
\begin{lemma}
\label{lemma:factorization_jordan}
 Let $k = (k_t \ ; \ t \geq 0)$ be a left-invariant Brownian motion on $K$ independent of $\widetilde{g}$, with time sped up by a factor $4$. Then the process $\widetilde{g} k = (\widetilde{g}_t k_t \ ; \ t \geq 0)$ is equal in law to the Brownian motion $(g_t\ ; \ t\ge 0)$ on $G$. 
\end{lemma}
\begin{proof} Let us write a short proof in our context of Jordan algebra. We begin exactly as in the proof of Lemma \ref{lemma:factorization_gk}. By reproducing the initial computation mutatis mutandis, using the definition of $\widetilde g$ in Eq. \eqref{eq:def_tildeg}, we have
 \begin{align*}
     d( \widetilde{g}_t k_t )
     & = ( \widetilde{g}_t k_t ) \circ \left[ \mathrm{Ad}_{k_t^{-1}}( 2 \sum_{i=1}^{\dim E} X_i dB_t^{(i)} + 2 p \Lrm(e)dt ) + 
     k_t^{-1} dk_t \right]\ .
 \end{align*}
 Now, because of Lévy's characterization of Brownian motion, and the fact that $\mathrm{L}(e)$ is central, the increments $$\mathrm{Ad}_{k_t^{-1}}( 2 \sum_{i=1}^{\dim E} X_i dB_t^{(i)} + 2 p \Lrm(e)dt )$$ have the same distribution as $2 \sum_{i=1}^{\dim E} X_i dB_t^{(i)} + 2 p \Lrm(e)dt$. At this stage, we understand that in order for the increment $k_t^{-1} dk_t$ in $\kfrak$ to have the same covariance of the increment in $\pfrak$, we need to speed up time by a factor $4$. Thus, $\tilde g k$ is indeed an instance of the left-invariant Brownian motion on $G$.
\end{proof}
We have now all the ingredients to prove Theorems \ref{thm:main_continuous} and \ref{thm:dufresne_continuous}.
 \begin{proof}[Proofs of Theorems \ref{thm:main_continuous} and \ref{thm:dufresne_continuous} ]
 Let $g = ( g_t = \widetilde{g}_t k_t \ ; \ t \geq 0)$. This is a left Brownian motion according to Lemma \ref{lemma:factorization_jordan}. Lemma \ref{lemma:cond_prod} implies that the Markov processes $(\lambda,\ell)$ and $\lambda$ are intertwined with the Kernel $\mathcal K$. Thus the Markov function Theorem \cite{rogers1981markov} can be invoked which yields both points of the Main Theorem \ref{thm:main_continuous}.
 Our Dufresne identity for Jordan algebras in Theorem \ref{thm:dufresne_continuous} remains obviously true for $g = ( g_t \ ; \ t \geq 0)$ because $e$ is fixed by the group $K$ so that
 $$ (g_t^*)^{-1}(e) = ( \widetilde{g}_t^*)^{-1}\left( (k_t^*)^{-1}(e) \right) = ( \widetilde{g}_t^*)^{-1}(e) \ .$$

 \end{proof}
 
 \begin{rmk}[Remark on the proof]
 The statements of our Theorems remain valid for any $g = \tilde{g}k$ with $k = (k_t \ ; \ t \geq 0)$ being a left Lévy process in $K$. It follows from the $\Ad(K)$-invariance in law of $\widetilde g$ and independence of its right increments. Of course, such a setting   departs from the setting of Rider--Valk\'o, as the process $g = \tilde{g}k$ is not necessarily a Brownian motion.
 \end{rmk}

\section{Acknowledgements}
The authors are grateful to Philippe Bougerol for fruitful discussions. They are both supported by the ANR project ``POAS: Probability on Algebraic Structures'', grant number ANR-24-CE40-5511 . R. C. is further supported by the Institut Universitaire de France (IUF). M.D. is further supported by the ANR project CORTIPOM ANR-21-CE40-0019.

\appendix

\section{An invariance principle in our context}
\label{section:stroock_varadhan}

In \cite[Theorem 2.4]{stroock1973limit}, Stroock--Varadhan give an invariance principle on Lie groups, with possibly non-homogenous increments. They also rely only on a local chart at the neighborhood of the identity. 

In our setting, their criterion takes a much simpler form due to the special structure of the random walk $\widetilde{g}$. For example, increments are homogeneous, in $P$ and $\exp_G: \gfrak \rightarrow G$ is global diffeomorphism once restricted to $\pfrak$. These elements considerably simplify the hypothesis needed. Also, for the purpose of clarity, only in this Appendix, we distinguish between the exponential maps $\exp_G: \gfrak \rightarrow G$ and $\exp_E: E \rightarrow E_+$. Likewise for the logarithmic maps $\log: U \subset G \rightarrow \gfrak$ and $\log_E: E_+ \rightarrow E$, with $U$ being a neighborhood of the identity on which the Lie group logarithm is defined.

\begin{thm}
\label{thm:adaptation_of_stroock_varadhan}
Consider i.i.d. increments $p^n_1$, $p^n_2$, $\dots$ in $P \subset G$ with $p^n_j = \Prm\left( w_j^n \right)$. And set for $t \geq 0$
$$ \tilde{g}_t = p^n_1 \dots p^n_{\lfloor nt\rfloor} \ ,$$
which is a random walk on $G$. If $p^{(n)} = \Prm\left( w^{(n)} \right)$ has the same law as one of the increments, we assume 
\begin{itemize}
\item the following asymptotic behavior of the mean
$$
    \lim_{n \rightarrow \infty} n \E \log_E w^{(n)} = m \in E \ ,
$$
\item the convergence in law
$$
    \sqrt{n} \log_E w^{(n)} \stackrel{\Lc}{\longrightarrow} X
$$
with $X$ having identity covariance,
\item with the moment $\E \| \sqrt{n} \log_E w^{(n)} \|^{2+\delta}$ remaining bounded, for a given $\delta>0$.
\end{itemize}
Then, for any finite time horizon $T>0$, $\tilde{g}$ converges in law in $\Dc( [0,T] , G)$ towards a diffusion with generator
$$
   2L(m) + 2 \sum_{i}^{\dim E} X_i^2 \ .
$$
\end{thm}
\begin{proof}[Pointers to proof]
In the notations of \cite{stroock1973limit}, we need a local chart around $\id_G$ given by a map $\phi$. In our particular case, because $\exp_G: \pfrak \rightarrow P$ and $\log_G: P \rightarrow \pfrak$ are global diffeomorphisms, we can take $\phi = \log_G$. Moreover, we know also that $\log_{E}: E_+ \rightarrow E$ is well-defined and bijective by functional calculus. Thanks to \cite[Prop. II.3.4]{faraut1994analysis}, we have
\begin{align}
\label{eq:magic}
\forall x \in E,\ 
\Prm\left[ \exp_E\left( x \right) \right] & = \exp_G\left( 2 \Lrm(x) \right) \ ,
\end{align}
so that
\begin{align}
\label{eq:magic2}
\log_G p^{(n)} & = 2 \Lrm\left( \log_E w^{(n)} \right) \ .
\end{align}

This allows to define an average increment $g_{n,j} \in G$ as in Eq. \cite[(2.8)]{stroock1973limit}
via
$ \phi( g_{n,j} ) = \E \phi(p_j^n) = \E \phi(p^{(n)}) $. 
Thanks to Eq. \eqref{eq:magic2}, that is to say $g_{n,j} = \exp_G\left( \E \log_G(p^{(n)}) \right) = \exp_G\left( 2 \Lrm[ \E \log_E w^{(n)} ] \right)$.
We also have a covariance $a_{n,j} \in \gfrak \times \gfrak$ as in Eq. \cite[(2.10)]{stroock1973limit} via
$$ a_{n,j} = \E \left( \phi(p_j^n) - \phi(g_{n,j}) \right) \otimes  \left( \phi(p_j^n) - \phi(g_{n,j}) \right) \ ,$$
which is also independent of $j$ and is nothing but the covariance of the random variable $\log_G(p^{(n)})$, denoted $\mathrm{Cov}\left( \log_G(p^{(n)}) \right)$ 

 On the one hand, the cumulated average increments is
$$
   g_{n,1} \dots g_{n,\lfloor nt\rfloor}
   = \exp_G\left( 2 \Lrm[ \E \log_E w^{(n)} ] \right)^{\lfloor nt\rfloor}
   = \exp_G\left( \lfloor nt\rfloor 2 \Lrm[ \E \log_E w^{(n)} ] \right) \ ,
$$
and the cumulated covariance is
$$
   a_{n,1} + \dots  + a_{n,\lfloor nt\rfloor}
   = \lfloor nt\rfloor \mathrm{Cov}\left( \log_G(p^{(n)}) \right) 
   = 4 \lfloor nt\rfloor \mathrm{Cov}\left( \Lrm \log_E(w^{(n)}) \right) \ .
$$

Given our hypotheses the cumulated average and the cumulated covariance do converge uniformly in $t \in [0,T]$, for any given finite time horizon $T>0$, so that Theorem 2.4 in \cite{stroock1973limit} is valid.

Finally, notice that the Lindeberg condition given in \cite[Eq. (2.6)]{stroock1973limit} 
$$
    \lim_{n \rightarrow \infty} \sum_{j=1}^{\lfloor nT \rfloor} \P\left( p^n_j \notin U \right) = 0 \ ,
$$
for any neighborhood $U$ of $\id_G$, is implied by the $(2+\delta)$-moment condition. Indeed, thanks to an $\varepsilon_U>0$ of room depending on $U$ and the Markov inequality 
\begin{align*}
    \sum_{j=1}^{\lfloor nT \rfloor} \P\left( p^n_j \notin U \right)
= & \ \lfloor nT \rfloor \ \P\left( \log_G p^{(n)} \notin \log U \right)\\
= & \ \lfloor nT \rfloor \ \P\left( 2 \Lrm\left( \log_E w^{(n)} \right) \notin \log U \right)\\
\leq & \ nT \ \P\left( \| \log_E w^{(n)} \|_E \geq \varepsilon_U \right)\\
\leq & \ nT \ \frac{\E \| \sqrt{n} \log_E w^{(n)} \|^{2+\delta}}{n^{1+\half \delta}\varepsilon_U} \stackrel{n \rightarrow \infty}{\longrightarrow} 0 \ .
\end{align*}
\end{proof}

\bibliographystyle{alpha}
\bibliography{Vilko.bib}

\end{document}